\documentclass[final]{siamltex}

\newtheorem{remark}[theorem]{Remark}

\usepackage{amsmath}
\usepackage{amssymb}
\usepackage{euscript}
\usepackage{graphicx}
\usepackage{color}
\usepackage[notref,notcite]{showkeys}

\def\II{{\mathbb I}}

\def\ZZ{{\mathbb Z}}

\def\NN{{\mathbb N}}
\def\RR{{\mathbb R}}
\def\TT{{\mathbb T}}
\def\IIs{{\mathbb I}^s}

\def\NNs{{\mathbb N}^s}

\def\RRs{{\mathbb R}^s}

\def\TTs{{\mathbb T}^s}
\def\ZZs{{\mathbb Z}^s}
\def\ZZsp{{\mathbb Z}^s_+}

\def\Krs{H_{s,r}(\TTs)}
\def\tKrs{K^{r,a}(\TTs)}
\def\tUr{U^{r,a}(\TTs)}

\def\tUrn{U^{r,a}_*(\TTs)}

\def\Kw{K^{r,a}(\IIs,w)}
\def\Uw{U^{r,a}(\IIs,w)}
\def\Lw{L_2(\IIs,w)}

\def\supp{\operatorname{supp}}

\def\TTs{{\mathbb T}^s}

\def\supp{\operatorname{supp}}

\def\endproof{\hfill\vbox{\hrule height0.6pt\hbox{\vrule height1.3ex%
width0.6pt\hskip0.8ex\vrule width0.6pt}\hrule height0.6pt}}
\newcommand{\proofof}[1]{\noindent{\it Proof of #1.} \ignorespaces}

\newlength{\fixboxwidth}
\newcommand{\cF}{\mathcal F}

\newcommand{\cL}{\mathcal L}

\newcommand{\cP}{\mathcal P}

\newcommand{\IN}{\mathbb N}

\newcommand{\IR}{\mathbb R}

\newcommand{\bd}{{\bf d}}
\newcommand{\bk}{{\bf k}}
\newcommand{\bm}{{\bf m}}

\newcommand{\bt}{{\bf t}}

\newcommand{\bx}{{\bf x}}
\newcommand{\by}{{\bf y}}
\newcommand{\bj}{{\bf j}}

\newcommand{\ind}{{n}}

\definecolor{dgreen}{rgb}{0,.6,0}

\title{New explicit-in-dimension estimates for the cardinality of high-dimensional hyperbolic crosses and approximation of functions having mixed smoothness
\thanks{Alexey Chernov acknowledges the support by the Hausdorff Center of Mathematics, University of Bonn, Germany. Dinh D\~ung's research is  funded by Vietnam National Foundation for Science and Technology Development (NAFOSTED) under  Grant No. 102.01-2012.15. The authors would like to thank Erich Novak for valuable remarks and comments.
}}

\author{Alexey Chernov\thanks{Corresponding author,
Department of Mathematics and Statistics University of Reading  Whiteknights, PO Box 220 Berkshire, UK RG6 6AX ({\tt a.chernov@reading.ac.uk}).}
\and Dinh D\~ung \thanks{Corresponding author, Vietnam National University, Hanoi, Information Technology Institute, 144 Xuan Thuy, Cau Giay, Hanoi, Vietnam ({\tt dinhzung@gmail.vn}).}}

\begin{document}

\maketitle

\begin{abstract}
We are aiming at sharp and explicit-in-dimension estimations of the cardinality of $s$-dimensional hyperbolic crosses where $s$ may be large, and applications in high-dimensional approximations of functions having mixed smoothness. In particular, we provide new tight and explicit-in-dimension upper and lower bounds for the cardinality of  hyperbolic crosses. We apply them to obtain explicit upper and lower bounds for Kolmogorov $N$-widths and $\varepsilon$-dimensions in the space $L_2(\TTs)$ of a modified Korobov class $\tUr$  parametrized by positive $a$ of $s$-variate periodic functions on the $s$-torus $\TTs := [-\pi,\pi]^s$ having mixed smoothness $r$, as a function of three variables $N,s,a$ and $\varepsilon, s,a$, respectively, when $N,s$ may be large, $\varepsilon$ may be small and $a$ may range from $0$ to $\infty$. Based on these results we describe a complete classification of tractability for the problem of $\varepsilon$-dimensions of $\tUr$. In particular, we prove the introduced exponential tractability of this problem for $a>1$.
All of these methods and results are also extended to high-dimensional approximations of non-periodic functions in the weighted space $L_2([-1,1]^s,w)$ with the Jacobi weight $w$  by Jacobi polynomials with powers in hyperbolic crosses.
\end{abstract}



\pagestyle{myheadings}
\thispagestyle{plain}
\markboth{ALEXEY CHERNOV AND DINH D\~UNG}{HD HYPERBOLIC CROSSES AND APPROXIMATION OF FUNCTIONS}

\section{Introduction} \label{sec:intro}

~\newline
The recent decades have been designated by an increasing interest in numerical approximation of problems in high dimensions, in particular problems involving high-dimensional input and output data depending on a large number $s$ of variables. 
They naturally appear in a vast number of applications in Mathematical Finance, Chemistry, Physics (e.g. Quantum Mechanics), Meteorology, Machine Learning, etc. 
Typically, a numerical solution of such problems to the target accuracy $\varepsilon$ demands for a high exponentially increasing computational cost $\varepsilon^{-\delta s}$ for some $\delta >0$, so that numerical computations even for a moderate values of $\varepsilon$ will result in an unacceptably large computation times and memory requirements. This phenomenon is called {\em the curse of dimensionality}, a term suggested by Bellmann \cite{Be57} (in the context of our paper term ``high-dimensional'' refers to the number of variables $s\gg1$). This consideration is true in general, but in some cases {\em the curse of dimensionality} can be overcome, particularly when the high-dimensional data belong to certain classes of functions having mixed smoothness. Such functions can be optimally represented by means of the {\em hyperbolic cross}  (HC) approximation. For example,
trigonometric polynomials with frequencies in HCs have been
widely used for approximating functions with a bounded mixed derivative or
difference. These classical trigonometric HC approximations date back to
Babenko \cite{Ba60a}. For further sources
on HC approximations in this classical context we refer to 
\cite{DD86,Te93} and the references therein. 
Later on, these terminologies were
extended to approximations by wavelets \cite{DeKoTe98, SU09}, by B-splines
\cite{DD10,SU10}, and specially to algebraic polynomial approximations where the power of algebraic polynomials approximants are in HCs.
\cite{C11,CS09}. HC approximations have applications
in quantum mechanics and PDEs \cite{Ys04,Ys05,Ys10, GH10b},
finance \cite{Ge07}, numerical solution of stochastic PDEs
\cite{C11,CS09,CS13,ST03a,ST03b}, and
data mining \cite{Ga01} to mention just a few
(see also the surveys \cite{BG04} and \cite{GeGr08} and the references
therein).

In traditional trigonometric approximations of functions having a mixed smoothness, there are two kinds of HC used as the frequency domain of approximant trigonometric polynomials: {\em continuous HCs}  
\begin{equation} 
G(s,T) := \bigg\{\bk \in \ZZs: \prod_{i=1}^s \max(|k_i|,1) \leq T\bigg\}
\end{equation}
and {\em step HCs} formed from the dyadic blocks $\square_\bj$
\begin{equation} 
G^*(s,T) := \bigg\{\bigcup \square_\bj: \bj \in \ZZsp, \ \sum_{i=1}^s j_i \le \log T\bigg\},
\end{equation}
where $\square_\bj :=  \{\bk \in \ZZs: \lfloor2^{j_i-1}\rfloor \le |k_i| < 2^{j_i}, \ i = 1,...,s\}$. (There are some modifications of these HCs for trigonometric approximations of functions having a mixed smoothness and zero mean value in each variable, see, for instance, \cite{DD86,Te93}.) These HCs having the cardinality $\asymp T \log^{s-1} T$, play an important role in computing asymptotic orders of various characteristics of optimal approximation such as $N$-widths and $\varepsilon$-dimensions for classes of periodic functions having mixed smoothness.  In this work we study approximations by trigonometric polynomials with  frequencies from modified HCs defined in \eqref{GammasTa-def} and \eqref{def[Gamma_pm(s,T,a)(2)]} below. We derive tight and explicit-in-dimension upper and lower bounds for the cardinality of these modified HCs and then apply them in study of Kolmogorov $N$-widths and $\varepsilon$-dimensions of a correspondingly  modified Korobov function class. 

Let us recall the concepts of Kolmogorov $N$-width \cite{Ko36} and its inverse $\varepsilon$-dimension, linear $N$-widths and Gel'fand $N$-widths being the most important approximation characterizations (particularly in high-dimensional approximation), which will be systematically studied in the present paper. Let $X$ be a normed space and $W$ a central symmetric subset in $X$. The Kolmogorov $N$-width $d_N(W,X)$ is defined by
\[
d_N(W,X)
:= \ 
\inf_{L_N} \ \sup_{f \in W} \ \inf_{g \in L_N} \|f - g\|_X,
\] 
where the outer inf is taken over all linear subspaces $L_N$ in $X$ of dimension $\le N$. 
The linear $n$-width $\lambda_n(W,X)$ \cite{Ti60} given by
\[
\lambda_n(W,X)
:= \ 
\inf_{\Lambda_n} \ \sup_{f \in W}  \|f - \Lambda_n(f)\|_X,
\] 
where the inf is taken over all linear operators $\Lambda_n$ in $X$ with rank $\le n$. 
The Gel'fand $N$-width $d^N(W,X)$ (see \cite{Ti76, P85}) is defined by
\[
d^N(W,X)
:= \ 
\inf_{L_{-N}} \ \inf \{\varepsilon \ge 0: W\cap L_{-N}  \subset \varepsilon \, U\},
\] 
where the outer inf is taken over all linear subspaces $L_{-N}$ in $X$ of codimension $\le N$, and $U$ is the unit ball in $X$. 
 There is a vast amount of literature on optimal approximations and these $N$-widths, see \cite{Ti76}, \cite{P85}, in particular, for $s$-variate function classes \cite{Te93}. 

In computational mathematics, the so-called $\varepsilon$-dimension
$n_\varepsilon = n_\varepsilon(W,X)$ is used to quantify the computational complexity. This approximation characteristic is defined as the inverse of $d_N(W,X)$.
In other words, the quantity $n_\varepsilon(W,X)$ is  the minimal number
$n_\varepsilon$ such that the approximation of $W$ by a suitably
chosen approximant $n_\varepsilon$-dimensional subspace $L$ in $X$ gives the approximation error 
$\le \varepsilon$ (see \cite{DD79}, \cite{DD80}, \cite{DD91}).  
The quantity $n_{\varepsilon}$ represents a special case of the information
complexity which is described by  the minimal number $n(\varepsilon,s)$ 
of information needed to solve the corresponding $s$-variate linear approximation problem of the identity operator within error
$\varepsilon$ (see \cite[4.1.4, 4.2]{NW08}). For further information on this topic we refer the interested reader to the surveys in monographs \cite{NW08,NW10} and the references therein. In what follows we address optimal
approximation of functions from the unit ball of a modified Korobov space $\tKrs$  in the Hilbert space $L_2(\TTs)$ for which the Komorogov $N$-width,  Gel'fand $N$-width and the linear $N$-width coincide (see \eqref{eq[three-widths]} below) and therefore, so do the quantities $n_{\varepsilon}$ and $n(\varepsilon,s)$. 

The task of an efficient numerical approximation (for example, numerical solution of a linear, high-dimensional elliptic PDE by the finite element method) raises naturally the question of the optimal selection of the approximation (finite element) subspace with $N$ degrees of freedom.
Recalling the above terminology, this reduces to 
the problem of optimal linear approximation in $X$ of
functions from $W$ by linear $N$-dimensional subspaces,  Kolmogorov $N$-widths $d_N(W,X)$ and $\varepsilon$-dimension $n_\varepsilon(W,X)$, where $W$ is a smoothness class of functions having in some sense more regularity than $X\supset W$.
In the present work, the regularity of the class $W$ will be measured by $L_2$-boundedness of mixed derivatives sufficiently of higher order. 
Finite element approximation spaces based on HC frequency domains are suitable for this framework \cite{DU13} (cf. also \cite{BG04}). 

As a model we will consider functions on $\RRs$ which are $2\pi$-periodic in each variable, as
functions defined on the $s$-dimensional torus 
$\TTs := [-\pi, \pi]^s$ for which the end points of the interval $[-\pi,\pi]$ are identified for each coordinate component.
The space $H^{r{\bf 1}}(\TTs) := H^r(\TT) \otimes \dots \otimes H^r(\TT)$ consists of all periodic functions whose mixed derivatives of order $r>0$ are $L_2$-integrable (i.e. having mixed smoothness  $r$).
For the unit ball 
$U^{r{\bf 1}}(\TTs)$ of space $H^{r{\bf 1}}(\TTs)$
the following
well-known estimates hold true:
\begin{equation} \label{ineq[d_N(H^{alpha},L)]}
\begin{split}
A(r,s) N^{- r} (\log N)^{r (s-1)}
\ \le \
d_N(U^{r{\bf 1}}, L_2(\TTs))
\ \le \
A'(r,s)N^{- r} (\log N)^{r (s-1)},
\end{split}
\end{equation} 
or equivalently,
\begin{equation} \label{ineq[n_e(H^{alpha{1}},L)]}
B(r,s) \varepsilon^{- 1/r} |\log \varepsilon|^{(s-1)/r }
\ \le \
n_{\varepsilon}(U^{r{\bf 1}}, L_2(\TTs))
\ \le \
B'(r,s)\varepsilon^{- 1/r} |\log \varepsilon|^{(s-1)/r },
\end{equation} 
Here, $A(r,s)$, $A'(r,s)$ and $B(r,s)$, $B'(r,s)$ denote certain constants depending on the smoothness $r$ and the dimension $s$ which are usually not
computed explicitly.
The inequalities \eqref{ineq[d_N(H^{alpha},L)]}
were proved by Babenko \cite{Ba60a} already in 1960 for the basic linear
approximation by continuous HC spaces of trigonometric polynomials.
These estimates are quite satisfactory if $s$ is small and fixed.   

In the recent work of Dinh D\~ung and Ullrich \cite{DU13}, 
$A(r,s)$, $A'(r,s)$ and $B(r,s)$, $B'(r,s)$ have been estimated from above and below explicitly in $s$ when $s$ is large. In their paper, the class $U^{r{\bf 1}}(\TTs)$ is redefined in terms of the traditional dyadic decomposition of the frequency domain. These estimations are based on an approximation by trigonometric polynomials with frequencies in step HCs $G^*(s,T)$ and explicit-in-dimension estimations of $|G^*(s,T)|$. 
However, the authors were able to estimate them only for very large $n \ge 2^s$ or very small 
$\varepsilon \le 2^{- \delta s}$, $\delta > 0$, (see \cite[Thms. 4.10, 4.11]{DU13}). This does not give a complete picture of the convergence rate in high-dimensional settings. 
The reason is that the step HC approximations of the class  $U^{r{\bf 1}}$ involve whole dyadic blocks of frequencies which have the cardilnality at least $2^s$.

In the present paper, to avoid this fact we suggest to replace $H^{r{\bf 1}}(\TTs)$ by another space $\tKrs$ which is defined as a modification of the well-known Korobov space, and construct appropriate continuous HCs for the trigonometric approximations of functions from this space. This will allow to derive very tight and explicit-in-dimension upper and lower estimates for the cardinality of continuous HCs and further sharp estimates for Kolmogorov $N$-widths and $\varepsilon$-dimensions. Along with the smoothness and dimension indices $r$ and $s$, the norms on spaces $\tKrs$ will be parametrized by a positive number $a>0$ allowing for a general and sharp estimation.

As usual, $L_2(\TTs)$ denotes the Hilbert space of functions on $\TTs$ equipped
with the inner product
\begin{equation}
(f,g) 
:= \
(2\pi)^{-s} \int_{\TTs} f(\bx) \overline{g(\bx)} \, d\bx.
\end{equation}
The norm in $L_2(\TTs)$ is defined as $\|f\| := (f,f)^{1/2}$. For $\bk \in \ZZ^s$, let
$\hat{f}(\bk): = (f, e_{\bk})$ be the $s$th Fourier coefficient 
of $f$, where $e_{\bk}(\bx):= e^{i(\bk,\bx)}$. 

Let us introduce spaces $\tKrs$.  For a given $r \ge 0$, $a>0$ and a
vector $\bk \in \ZZs$ we define the scalar $\lambda_a(\bk)$ by
\begin{equation*}
\lambda_a(\bk)
 := \
\prod_{j = 1}^s \, \lambda_a(k_j),
\quad
\lambda_a(k_j)
 := \
(a + |k_j|)
\end{equation*}
and the Korobov function $\kappa_{r,a}$ (in distributional sense) at $\bx\in\TTs$ by
the relation
\begin{equation}
\kappa_{r,a}(\bx) \ := \ \sum_{\bk \in \ZZs}
\lambda_a(\bk)^{-r} \, e_{\bk}({\bx}).
\end{equation}
Then the Hilbert space $\tKrs$ is defined as
\begin{equation}
\tKrs:=\{f: f =
\kappa_{r,a}*g, \ g \in L_2(\TTs)\}
\end{equation} 
with the norm
$\|f\|_{\tKrs} \ := \|g\|$,
where $f*g$ denotes the standard convolution of functions $f$ and $g$ on $\TTs$.

The notion of space $\tKrs$ is a modification of the notion of the classical Korobov space. 
For $r > 1/2$, the kernel $K_a$ defined at $\bx$ and
$\by$ in $\TTs$ as $K_a(\bx,\by):=\kappa_{2r,a}(\bx - \by)$
is the reproducing kernel for the Hilbert space $\tKrs$. For a definitive treatment of
reproducing kernel, see, for example, \cite{Ar50}. The linear span of the set of functions 
$\{\kappa_{r,a}(\cdot - \by): \by \in \TTs\}$ is dense in $\tKrs$. In the language of Machine Learning, this means that
the reproducing kernel for the Hilbert space is universal. In the recent paper \cite{DM13} some upper and lower bounds of  multivariate approximation by translates of the Korobov function on sparse Smolyak grids has been established.

A similar notion of generalized Korobov space $\Krs$ was introduced in \cite[A.1, Appendix]{NW08}. This space is defined in the same way as the definition of $\tKrs$ by replacing the scalar $\lambda_a(\bk)^r$ 
by the scalar $\varrho_{s,r}(\bk)$ depending on two parameters $\beta$ and $\beta_1$.
Korobov spaces and their modifications are important for the study of approximation and computation problems of smooth multivariate periodic functions, especially in high-dimensional settings. For further information,
see detailed surveys  and references in the books \cite{Te93}, \cite{SJ94}, \cite{NW08}. 

Note that the spaces $H^{r{\bf 1}}(\TTs)$, $\tKrs$ and $\Krs$
coincide as function spaces equipped with equivalent norms.
However, if $s$ is large the unit balls with respect to the norms of these spaces differ
significantly. 
As will be shown in the present paper, for  the space $\tKrs$, the scaling parameter $a$ defining different equivalent norms, controls the ``size'' of the unit ball of $\tKrs$ as a subset in $L_2(\TTs)$, and determines crucially different high-dimensional approximation properties of functions from $\tKrs$.

Let $\tUr$ be the function class defined as the unit ball in $\tKrs$. 
In the present paper, we  compute $d_N(\tUr,L_2(\TTs))$ and $n_\varepsilon(\tUr,L_2(\TTs))$ explicitly in $N,s,a$ and $\varepsilon,s,a$, respectively. The core of our theory in both, periodic and non-periodic settings, is based upon sharp cardinality estimates for the index sets $\Gamma(s,T,a)$ and $\Gamma_\pm(s,T,a)$ defined as follows. 
For $s \in \NN$ and $a > 0$ and $T>0$, we consider the sets of multi-indices $\bk = (k_1,\dots, k_s)$ and set
\begin{equation} \label{GammasTa-def}
\Gamma(s,T,a) := \ \bigg\{\bk \in \NN_0^s: \prod_{i=1}^s(k_i+a) \leq T\bigg\}
\end{equation}
and
\begin{equation} \label{def[Gamma_pm(s,T,a)(2)]}
\Gamma_\pm(s,T,a) := \ \bigg\{\bk \in \ZZs: \prod_{i=1}^s(|k_i|+a) \leq T\bigg\}
\end{equation}
which are referred as \emph{a corner and symmetric continuous HC}, or shortly {\em HC}.

Denote by $|G|$ the cardinality of the set $G$. Notice that the problem of computing $|\Gamma(s,T,a)|$ and 
$|\Gamma_\pm(s,T,a)|$ in our setting is itself interesting as a problem of a classical direction in Number Theory investigating the number of the integer points in a domain such as a ball and a sphere \cite{Vi63}, \cite{Ch63}, 
\cite{HB99}, a hyperbolic domain \cite{Le96}, \cite{DR98}, \cite{Du88}, \cite{Go84}, \cite{DD84}, etc.  Specially, in 
\cite{DD84} the convergence rate of cardinality of the intersection of HCs was computed and applied to estimations of $d_N(W,L_2(\TTs))$ where $W$ is a class of several $L_2(\TTs)$-bounded mixed derivatives. 

Motivated by all the above arguments,  the main goal of this paper to prove upper and lower bounds for $|\Gamma(s,T,a)|$ and $|\Gamma_\pm(s,T,a)|$ in a {\em new high-dimensional approach} as a functions of three variables $s,T,a$ when $s,T$ may be large and $a$ ranges from $0$ to $\infty$, and then to apply to the problem of estimations of $d_N(\tUr,L_2(\TTs))$ and $n_\varepsilon(\tUr,L_2(\TTs))$. In fact, as shown in this paper this problem is reduced to estimations of $|\Gamma(s,T,a)|$ and $|\Gamma_\pm(s,T,a)|$ and their inverse. More precisely, we stress on finding the accurate dependence of $|\Gamma(s,T,a)|$ and $|\Gamma_\pm(s,T,a)|$ as a function of three variables $s$, $T$ and $a$. Although $T$ is the main parameter in the study of convergence rates with respect to $T$ going to infinity, the parameters $s$ and $a$ may seriously affect this rate when $s$ is large or the positive parameter $a$ ranges through the critical value $1$. Our method of estimation is based on comparison $|\Gamma(s,T,a)|$ and $|\Gamma_\pm(s,T,a)|$ with the volume of smooth HCs $P(s,T,a)$ (see a definition in \eqref{P-def}) and  tight non-asymptotic estimates of the latter. A new crucial element in our volume-based estimation approach is that the cardinalities of $\Gamma(s,T,a)$ and $\Gamma_\pm(s,T,a)$ are compared with volumes of shifted smooth HCs properly smaller than the smallest smooth HC containing the domain $Q(s,T,a)$ (see a definition in \eqref{Q-def}) having the volume equal to $|\Gamma(s,T,a)|$ (see Section \ref{sec:def+pre}). 
 This shift is made possible by introduction of the parameter $a$ which is new in the present work.

As a by-product of our analysis, we prove
that the volume of smooth HCs $P(s,T,a)$ can be reduced to the $s$th remainder of the Taylor series of $\exp(-t)$, which can be tightly estimated by 
\begin{equation} \nonumber
\frac{t^s}{(s-1)!(t+s)}  < (-1)^s \sum_{k=s}^\infty \frac{(-t)^k}{k!} < \frac{t^s}{(s-1)!(t+s-1)}, 
\quad s \ge 1, \ t > 0.
\end{equation}
To the knowledge of the authors, so far these estimates have been unknown. The establishing of them as well of estimates of $|\Gamma(s,T,a)|$ and $|\Gamma_\pm(s,T,a)|$ by the volume-based estimation approach requires to overcome certain technical difficulties.

The obtained results are then utilized in derivation of tight upper and lower bounds for $d_N(\tUr,L_2(\TTs))$ and $n_\varepsilon(\tUr,L_2(\TTs))$ explicit in three parameters $n,s,a$ or $\varepsilon,s,a$, as well in study of tractabilities of the problem of $n_\varepsilon(\tUr,L_2(\TTs))$. 
As mentioned above in a Hilbert space $X$, the Kolmogorov $N$-widths $d_N(W,X)$ and linear $N$-widths $\lambda_N(W,X)$ coincide, but the Gel'fand $N$-widths $d^N(W,X)$ can be smaller. However, from  
\cite[Theorem 2.2, p. 65]{P85} we have the relations 
\begin{equation} \label{eq[three-widths]}
d_N(\tUr, L_2(\TTs)) 
\ = \ 
d^N(\tUr, L_2(\TTs))
\ = \ 
\lambda_N(\tUr, L_2(\TTs)).
\end{equation}
Therefore, for our purpose it is sufficient to study the quantities
\begin{equation} \label{dNne-def}
d_N:= \ d_N(\tUr,L_2(\TTs)) \ \text{and} \ n_\varepsilon:= \ n_\varepsilon(\tUr,L_2(\TTs)).
\end{equation}
We emphasize that in general (without any additional restriction on functions to be approximated) the problem of approximation of infinite differentiable multivariate functions is intractable \cite{NW09a}. In the present paper, we introduce a new notion of tractability, 
{\em exponential tractability} and study it together with well-known tractabilities of the problem of $n_\varepsilon$. The scaling parameter $a$ is the key to the classification of tractabilities of the problem of $n_\varepsilon$ and to the control of the computational complexity. We show that  this problem for the absolute error criterion is exponentially tractable if $a>1$  weakly tractable but polynomially intractable if $a=1$, and intractable if $a<1$. More precisely, for every $q \in [2,\infty)$ satisfying 
$\lambda:= a - 2/q > 0$, and every 
$\varepsilon \in (0,1]$ and $N \in \NN$, we have  
\begin{equation} \label{introduction[n_e<]}
n_\varepsilon 
\  \leq \ 
q \lambda^{-qs}\varepsilon^{-(1+q)/r},
\end{equation}  
and correspondingly,
\begin{equation} \label{introduction[d_N<]}
d_N 
\  \leq \ 
2^rq^{r/(1 + q)}\lambda^{-[qr/(1+q)]s}N^{-r/(1+q)}.
\end{equation}
In the case of exponential tractability $a>1$, both the $N$-widths and $\varepsilon$-dimensions are decreasing exponentially in dimension $s$ when $s$ going to $\infty$. Moreover, the estimates \eqref{introduction[n_e<]} and 
\eqref{introduction[d_N<]} tell us that we can control the computational complexity of a high-dimensional approximation or reach a desirable accuracy of approximation keeping balance between three parameters $s,a,\varepsilon$ or $s,a,N$. We also prove that  the problem of $n_\varepsilon(a^{rs}\tUr,L_2(\TTs))$ for the normalized error criterion is weakly tractable but polynomially intractable for every $a>0$. The tractability of linear approximation problem for the generalized Korobov space $\Krs$ for various values of parameters $\beta_1$ and $\beta$ was studied in \cite[Pages 184--185]{NW08}. 

In the next step of our investigation, these estimates will be sharpened by other upper and lower bounds for $n_\varepsilon$ and $d_N$. In particular, 
for $r >  0$, $s \ge 2$, $a > 1/2$, we prove for every $\varepsilon \in (0,[a-1/2]^{-sr})$,
\begin{equation} \nonumber
n_\varepsilon 
\  \leq \ 
\frac{2^s \varepsilon^{-1/r} (\ln \varepsilon^{-1/r} - s\ln (a-1/2))^s }
{(s-1)!\big(\ln \varepsilon^{-1/r} + s\big(1-\ln (a-1/2)) - 1\big)},
\end{equation}
and  for every  $\varepsilon \in (0,[a+1/2]^{-sr})$,
\begin{equation} \nonumber
n_\varepsilon 
\  \geq \ 
\frac{2^s \varepsilon^{-1/r}  (\ln \varepsilon^{-1/r}  - s\ln (a+1/2))^s}
{(s-1)!(\ln \varepsilon^{-1/r}  + s(1-\ln (a + 1/2))} \ - \ 1,
\end{equation}
and corresponding upper and lower bounds for $d_N$ (Theorem \ref{theorem[<d_N<]}). 

All of these methods and results are extended to HC approximations  of functions from the non-periodic modified Korobov space $\Kw$ in the weighted space $\Lw$ with the Jacobi weight $w$ by Jacobi polynomials with powers in the corner HC $\Gamma(s,T,a)$, where $\IIs:= [-1,1]^s$.

In brief, the paper is organized as follows. In Section \ref{sec:def+pre}, we give some preliminary estimates of $|\Gamma(s,T,a)|$ and $|\Gamma_\pm(s,T,a)|$ via the volume of corresponding smooth HCs. 
In Section 
\ref{sec:non-asymp}, we prove the non-asymptotic estimates for $|\Gamma(s,T,a)|$ and $|\Gamma_\pm(s,T,a)|$.
 In Section \ref{Tractabilities}, we prove some upper estimates for $d_N$ and  $n_\varepsilon$, and investigate tractabilities of the problem of  $n_\varepsilon$. 
In Section \ref{widths and dimensions}, we  give  sharpened  upper and lower bounds for $n_\varepsilon$ and $d_N$.
In Section \ref{Non-periodic} we extend the methods and results for periodic approximations to non-periodic  approximations.

\section{Preliminary estimates via the volume of smooth HCs} \label{sec:def+pre}

~\newline
A natural way to estimate $|\Gamma(s,T,a)|$ and $|\Gamma_\pm(s,T,a)|$ is to compare them with the volume of the corresponding {\em corner smooth HC}
\begin{equation}\label{P-def}
P(s,T,a') :=  \big\{\bx \in \RR^s_+: \prod_{j=1}^s (x_j+a') \leq T \big\}
\end{equation}
which is a bounded domain in $\RRs$,  where $a'$ may be not equal to $a$. 

Consider the  set
\begin{equation}\label{Q-def}
Q(s,T,a) := \bigcup_{\bk \in \Gamma(s,T,a)} (\bk + [0,1]^s).
\end{equation}
Obviously, it holds that
\begin{equation} \label{Gamma=IntQ}
|\Gamma(s,T,a)| = \int_{Q(s,T,a)} \bd \bx.
\end{equation}
Using this equation we will compare $|\Gamma(s,T,a)|$ with the volume of $P(s,T,a)$
\begin{equation}\label{I-def}
I(s,T,a') := \int_{P(s,T,a')} \bd \bx.
\end{equation}

Put $\lfloor \bx \rfloor := (\lfloor x_1 \rfloor,..., {\lfloor x_s \rfloor})$ for $\bx \in \RRs$ where 
$\lfloor t \rfloor$ denotes the integer part of $t \in \RR$. The following lemma gives rough bounds of $|\Gamma(s,T,a)|$ via the volume $I(s,T,a)$ based on direct set inclusions.

\begin{lemma} \label{lem:PQ-incl}
For every $s \in \IN$, $T>0$, and $a>0$, there hold the inclusions
\begin{equation} \label{PQ-incl}
Q(s,T,a+1) 
\ \subsetneq \ 
P(s,T,a)
\ \subsetneq \ 
Q(s,T,a) 
\end{equation}
and consequently,
\begin{equation}\label{IGamma-bnd1}
|\Gamma(s,T,a+1)|
\ < \ 
I(s,T,a)
\ < \ 
|\Gamma(s,T,a)|.
\end{equation}
\end{lemma}

\begin{proof}
We observe that $\bx \in Q(s,T,a)$ if and only if $\lfloor \bx \rfloor \in \Gamma(s,T,a)$. Therefore the relation
\begin{equation} \nonumber
T \geq
\prod_{j=1}^s\big(\lfloor x_j \rfloor + a +1\big) \geq
\prod_{j=1}^s\big(x_j + a\big) \geq
\prod_{j=1}^s\big(\lfloor x_j \rfloor + a\big) 
\end{equation}
implies \eqref{PQ-incl} with sharp inclusions and, by \eqref{Gamma=IntQ} and \eqref{I-def}, the inequalities \eqref{IGamma-bnd1}.
\end{proof}

The next lemma generalizes and sharpens the left inequality in \eqref{IGamma-bnd1}.

\begin{lemma}\label{lemma[Gamma<(2)]}
Suppose $0 < \delta \leq 1$. Then for every $s \in \IN$, $T \geq \delta^s$, and $a > \delta$, it holds that
\begin{equation}\label{ineq[Gamma<(2)]}
 |\Gamma(s,T,a)| \ < \ (1/\delta)^s I(s,T,a - \delta),
\end{equation}
and 
\begin{equation}\label{ineq[Gamma_pm<(2)]}
 |\Gamma_\pm(s,T,a)| \ < \  (2/\delta)^s I(s,T,a - \delta).
\end{equation}
\end{lemma}

\begin{proof} Since $ |\Gamma_\pm(s,T,a)| <  2^s|\Gamma(s,T,a)|$, it is sufficient to prove \eqref{ineq[Gamma<(2)]}. For $\delta = 1$, estimate  \eqref{ineq[Gamma<(2)]} is equivalent to the left inequality in \eqref{IGamma-bnd1} if changing $a$ to $a+1$. For $0 < \delta < 1$ we introduce a $(1-\delta)$-shifted set
\begin{equation}
 \tilde Q(\delta) := \tilde Q(\delta,s,T,a) := \{\bx ~:~ x_j = y_j - (1-\delta), \by \in Q(s,T,a)\}
\end{equation}
(we suppress the dependence on $s,T,a$ to simplify the notations) and its subset
\begin{equation}
 \tilde Q^+(\delta) := \{\bx \in \tilde Q(\delta) ~:~ x \geq 0\}.
\end{equation}
For $y_j = x_j + (1-\delta)$ we observe
\begin{equation}
T \geq \prod_{j=1}^s (\lfloor y_j \rfloor +a)
\geq \prod_{j=1}^s (x_j  + a -\delta),
\end{equation}
implying $\tilde Q^+(\delta) \subsetneq P(s,T,a-\delta)$ and therefore
\begin{equation}
|\tilde Q^+(\delta)| \leq I(s,T,a-\delta).
\end{equation}
Let $e \in\{0,1\}^s$ and consider
\begin{equation}
\tilde Q_e(\delta) = \{\bx \in \tilde Q(\delta) ~:~ x_j < \delta \text{ if } e_j = 0 \wedge x_j \geq \delta \text{ if } e_j = 1 \}
\end{equation}
and correspondingly, $\tilde Q_e^+(\delta) = \tilde Q_e(\delta) \cap \tilde Q^+(\delta)$. This defines disjoint decompositions
\begin{equation}
\tilde Q(\delta) = \bigcup_{e \in \{0,1\}^s}
\tilde Q_e(\delta),
\quad
\tilde Q^+(\delta) = \bigcup_{e \in \{0,1\}^s}
\tilde Q_e^+(\delta).
\end{equation}
For volumes of $\tilde Q_e(\delta)$ and $\tilde Q_e^+(\delta)$ we obviously have the relations
\begin{equation}
 \frac{|\tilde Q_e(\delta)|}{|\tilde Q_e^+(\delta)|} =  
\frac{1}{\delta^{|e|}}, \quad \mbox{where} \quad |e| := \sum_{j=1}^s e_j.
\end{equation}
Therefore,
\begin{equation}
\begin{split}
|Q(s,T,a)|&= 
|\tilde Q(\delta)| = \sum_{e \in \{0,1\}^s} |\tilde Q_e(\delta)|\\ 
&= \sum_{e \in \{0,1\}^s} \delta^{-|e|} |\tilde Q_e^+(\delta)|
\leq
\delta^{-s} I(s,T,a-\delta).
\end{split}
\end{equation}
\end{proof}

Due to the specific geometrical structure of the symmetric HC $\Gamma_\pm(s,T)$, the upper bound 
\eqref{ineq[Gamma_pm<(2)]} can be improved for $\delta = \frac{1}{2}$ as in the following lemma.

\begin{lemma}\label{lem:IGamma-sym}
For every $s \in \IN$ $a > 1/2$ and $T \ge 1$, it holds that
\begin{equation}\label{IGamma-sym-bnd}
 2^sI(s,T, a + \tfrac{1}{2})  < \ |\Gamma_\pm(s,T,a)| \ < \ 2^sI(s,T,a - \tfrac{1}{2}).
\end{equation}
\end{lemma}

\begin{proof} 
We set 
\begin{equation}
 Q_\pm(s,T) := \bigcup_{\bk \in \Gamma_\pm(s,T,a)} \big\{ \bx \in \RR^s:  x_j  \in [k_j-\tfrac{1}{2},k_j+\tfrac{1}{2}), \ j=1,\dots, s \big\}.
\end{equation}
Then we have
\begin{equation}
|\Gamma_\pm(s,T)| = \int_{Q_\pm(s,T)} \bd\bx.
\end{equation}
The relations
\begin{equation}
T \geq \prod_{j=1}^s (|x_j|+ a +\tfrac{1}{2}) 
\geq \prod_{j=1}^s (|x_j \pm \tfrac{1}{2}|+a)
\geq \prod_{j=1}^s (|x_j|+ a - \tfrac{1}{2}) 
\end{equation}
imply 
\begin{equation}
\big\{ |\bx| \in P(s,T,a + \tfrac{1}{2}) \big\} \,\subsetneq \, Q_\pm(s,T,a)\, 
\subsetneq \, \big\{ |\bx| \in P(s,T,a - \tfrac{1}{2}) \big\},
\end{equation}
and consequently, by symmetry \eqref{IGamma-sym-bnd}.
\end{proof}


Unfortunately, we have no analogue of Lemma \ref{lem:IGamma-sym} for the corner HC $\Gamma(s,T,a)$. 
We able to establish only the upper bound $|\Gamma(s,T,a)| \ < \ I(s,T,a - \tfrac{1}{2})$ for $T$ large enough. Namely, we have the following theorem.




\begin{theorem} \label{thm:Gamma-asy}
For any $s \in \NN$, and $a > \tfrac{1}{2}$, there exists $T_* = T_*(s,a) >0$ such that
\begin{equation} \label{Gamma-asy}
 |\Gamma(s,T,a)| \leq I(s,T,a - \tfrac{1}{2}), \quad
\forall T \geq T_*(s,a).
\end{equation}
\end{theorem}

The proof of Theorem \ref{thm:Gamma-asy} is given in Appendix in Section \ref{appendix}.

 Estimations of $I(s,T,a)$ are addresses in the next section.



\section{Non-asymptotic bounds for the volume of  smooth HCs} \label{sec:non-asymp}

~\newline
The inequalities \eqref{IGamma-bnd1}--\eqref{ineq[Gamma_pm<(2)]}, \eqref{IGamma-sym-bnd} and \eqref{Gamma-asy} allow us to estimate $|\Gamma(s,T,a)|$ and $|\Gamma\pm(s,T,a)|$ by the volume the smooth HC 
$I(s,T,a')$, which is, as a matter of fact, a simpler task.
Indeed, below we will show that the integral $I(s,T,a)$ for any $a>0$ can be represented as a sum of an infinite series. This series is related to the remainder of the Taylor series of the exponential $\exp(-t)$, $t > 0$. 

\subsection{Tight non-asymptotic bounds for the volume of smooth HCs}\label{subsec:non-asymp-form}

~\newline
To formulate the result,  for $s \in \NN_0$, we introduce the function
\begin{equation}\label{F-def}
F_s(t) := \ (-1)^s \sum_{n=s}^\infty(-1)^{n} p_n(t), \ \ t \in \RR,
\end{equation}
where $p_s(t) := t^s/s!$. Observe that $\exp(-t)=F_0(t)$ and $F_s(t)$ is the absolute value of the $s$th remainder of the Taylor series of the exponential $\exp(-t)$.

The following lemma shows that the problem of estimations of the volume of smooth HC $I(s,T,a)$ can be reduced to estimations of $F_s(t)$.
\begin{lemma}\label{lem:I-val}
Suppose $T >0$, $s \in \NN$ and $a>0$. Then it holds that
\begin{equation}\label{I-val}
 I(s,T,a) = \left\{ 
\begin{array}{cc}
0, & T \leq a^s, \\[2ex] \displaystyle
T F_s(\ln T - s \ln a), & T > a^s.
\end{array}
\right.
\end{equation}

\end{lemma}
\begin{proof}
The continuous HC $P(s,T,a)$ is empty if $T < a^s$ and contains only zero if $T = a^s$. This yields the first  line in 
\eqref{I-val}. Suppose now that $T > a^s$. From \cite[Lemma 3]{DR98} we know that
\begin{equation}
I(s,T,a) = (-1)^{s+1}(T-a^s) + T\sum_{n=1}^{s-1} \frac{(\ln T - s \ln a)^n(-1)^{s-1-n}}{n!}.
\end{equation}
Then for $t := \ln T - s \ln a > 0$, there holds
\begin{equation}\label{I/T-expr}
\begin{split}
\frac{1}{T}I(s,T,a) &= (-1)^{s+1} (1 - e^{-t}) + \sum_{n=1}^{s-1} \frac{t^n(-1)^{s-1-n}}{n!} \\
&= (-1)^s \left( e^{-t} - \sum_{n=0}^{s-1} \frac{t^n(-1)^{n}}{n!} \right)\\
&= (-1)^s \sum_{n=s}^\infty (-1)^{n}\frac{t^n}{n!} \\
&= F_s(t).
\end{split}
\end{equation}
This yields the second line in \eqref{I-val}. 
\end{proof}

Note that for $s = 0$ it formally holds
\begin{equation}
 I(0,T,a) = T F_0(\ln T)= 1.
\end{equation}

Relation \eqref{I-val} is \emph{exact} but involves an infinite summation on the right-hand side, whose behavior is not clearly seen from \eqref{F-def}. To make it explicit, we prove very tight bounds for the series $F_{s}(t)$ in terms of the power $p_{s-1}$. This approximation is somewhat surprisingly good, as we observe from the following \emph{non-asymptotic} estimate with explicit constants.

\begin{theorem}\label{thm:errTaylor}
 For any $s \in \NN$ and $t>0$, the following estimate holds true
\begin{equation}\label{errTaylor}
\frac{t}{t+s} p_{s-1}(t) < F_s(t) < \frac{t}{t+s-1} p_{s-1}(t).
\end{equation}
\end{theorem}

The proof of Theorem \ref{thm:errTaylor} will be given in Subsection \ref{subsec:Taylor}.

This theorem and Lemma \ref{lem:I-val} imply directly

\begin{corollary}\label{cor:I-bnd}
For $s \in \IN$, $a>0$ and $T>a^s$, it holds that 
\begin{equation}\label{I-bnd}
\frac{T (\ln T - s\ln a)^s }{(s-1)!(\ln T + s(1-\ln a))}
\ < \
I(s,T,a)
\ < \
\frac{T (\ln T - s\ln a)^s }{(s-1)!(\ln T + s(1-\ln a) - 1)}.
\end{equation}
\end{corollary}

A combination of Theorem \ref{thm:Gamma-asy} and Corollary \ref{cor:I-bnd} where $a$ is substituted with $a - \frac{1}{2}$ implies the following bound.

\begin{corollary}  \label{cor:IGamma-newbnd}
For any $s \in \NN$ and $a > \tfrac{1}{2}$, there exists $T_* = T_*(s,a) >0$ such that
\begin{equation}\label{IGamma-newbnd}
\begin{split}
|\Gamma(s,T,a)| 
&\ < \
\frac{T \big(\ln T - s\ln (a - \tfrac{1}{2})\big)^s }{(s-1)!\big(\ln T + s\big(1-\ln (a - \tfrac{1}{2})\big) - 1\big)}.
\end{split}
\end{equation}
\end{corollary}

\begin{corollary}\label{corollary[Gamma<(3)]}
For every $s \in \IN$, $a > 0$, $0 < \delta \le 1$, $\delta <a$ and $T > 0$, it holds that
\begin{equation}\label{ineq[Gamma<(3)]}
 |\Gamma(s,T,a)| \ < \  \delta^{-1} T^{1 + 1/\delta} (a - \delta)^{-s/\delta},
\end{equation}
and 
\begin{equation}\label{ineq[Gamma_pm<(3)]}
 |\Gamma_\pm(s,T,a)| \ < \  2\delta^{-1} T^{1 + 2/\delta} (a - \delta)^{-2s/\delta}
\end{equation}
\end{corollary}

\begin{proof}
Let us prove the inequality\eqref{ineq[Gamma<(3)]} in the lemma. The other one can be proven in a similar way. 
Since $|\Gamma(s,T,a)| = 0$ for $0 < T < a^s$, it is enough to consider the case where $T > (a - \delta)^s$. By Lemmas \ref{lemma[Gamma<(2)]}, \ref{lem:I-val} and Theorem \ref{thm:errTaylor} we have that
 \begin{equation}
\begin{split}
|\Gamma(s,T,a)| 
 &\le   \delta^{-s} I(s,T,a - \delta) \\
\ &< \  \delta^{-s}  
\frac{T [\ln T - s\ln (a - \delta)]^s }{(s-1)!(\ln T + s[1-\ln (a - \delta)] - 1)} \\
\ &< \ 
\delta^{-1} T \frac{[\delta^{-1}(\ln T - s\ln (a - \delta)]^{s-1}}{(s-1)!} \\
\ &< \ \delta^{-1} T
 \operatorname{exp}[\delta^{-1}(\ln T - s\ln (a - \delta)] \\
 \ &= \ \delta^{-1} T^{1 + 1/\delta} (a - \delta)^{-s/\delta}.
\end{split}
\end{equation}
\end{proof}

Corollary \ref{corollary[Gamma<(3)]} shows that if $a > 1$ and $0<\delta \le 1$ are any fixed numbers such that 
$\lambda := a - \delta > 1$, then the number of integer points in the hyperbolic cross $\Gamma(s,T,a)$ and $\Gamma_\pm(s,T,a)$ is decreasing exponentially as $\lambda^{-s/\delta}$ and $\lambda^{-2s/\delta}$ respect to $s$ when $s$ going to $\infty$.  It will be used in study of HC approximations, $N$-widths and  tractabilities of the problem of $\varepsilon$-dimensions in Sections \ref{Tractabilities}--\ref{Non-periodic}. 

\subsection{Proof of Theorem \ref{thm:errTaylor}}
\label{subsec:Taylor}

~\newline

We define
\begin{equation}\label{def-aFh}
h_s(t) := \frac{F_{s+1}(t)}{p_s(t)}.
\end{equation}
Observe that Theorem \ref{thm:errTaylor} is equivalent to the following statement. For any $s \in \NN$ and $t>0$, we have 
\begin{equation} \label{errTaylor-h}
 \frac{t}{t+s}  < h_{s-1}(t) < \frac{t}{t+s-1}.
\end{equation}
Therefore, to prove Theorem \ref{thm:errTaylor} we will verify \eqref{errTaylor-h}.
The proof of \eqref{errTaylor-h} requires some auxiliary lemmas.

Note that \eqref{errTaylor} and \eqref{errTaylor-h} are  claimed for a fixed couple $(s,t)$  and are therefore,  non-asymptotic error estimates. Thus, we may assume that $t$ is a fixed positive real number and write $p_s, F_s, h_s$ instead of $p_s(t), F_s(t), h_s(t)$ if possible, in order to simplify and shorten the notations.

\begin{lemma} \label{lem:F-pos}
For every $t > 0$ and $s \in \NN_0$, we have
\begin{equation} \label{F-pos}
 0 < F_s(t) < \infty.
\end{equation} 
\end{lemma}

\begin{proof}
Suppose $s \geq t$. Then 
\begin{equation}
 \frac{p_n}{p_{n+1}} = \frac{n+1}{t} \geq \frac{s+1}{t} > 1, \qquad \forall n \geq s,
\end{equation}
i.e. the sequence $\{p_n\}_{n\geq s}$ converges monotonously to zero. Thus
\begin{equation} \label{F(s,t)=[t<s]}
 F_s = \underbrace{(p_{s}-p_{s+1})}_{>0} + \underbrace{(p_{s+2}-p_{s+3})}_{>0} + \dots > 0
\end{equation}
and \eqref{F-pos} follows. In order to analyze the case $t > s$ we observe the identity
\begin{equation}
(-1)^s F_s = \exp(-t) - \sum_{n=0}^{s-1} (-1)^{n} \frac{t^n}{n!}
\end{equation}
and the relation
\begin{equation}
 \frac{p_{n-2}}{p_{n-1}} = \frac{n-1}{t} \leq \frac{s-1}{t} < 1, \qquad 2 \leq n \leq s,
\end{equation}
i.e. the finite sequence $\{p_n\}_{n=0}^{s-1}$ is monotonously increasing.
If $s$ is even, we have
\begin{equation}
F_s(t) = \underbrace{\exp(-t)}_{>0} + \underbrace{(-p_0 + p_1)}_{> 0} + \dots + \underbrace{(-p_{s-2} + p_{s-1})}_{> 0} > 0.
\end{equation}
If $s$ is odd we regroup the terms and obtain
\begin{equation}
-F_s(t) = \underbrace{(\exp(-t) - p_0)}_{< 0} + \underbrace{(p_1 - a)}_{<0} + \dots  +\underbrace{(p_{s-2}-p_{s-1})}_{< 0} < 0.
\end{equation}
The proof is complete.
\end{proof}

For every $t>0$ and $s \in \NN$, definition \eqref{def-aFh} implies the identities
\begin{equation}\label{ids-aFh}
p_s(t) = F_s(t) + F_{s+1}(t), \qquad h_s(t) = 1 - \frac{F_s(t)}{p_s(t)}, \qquad p_s(t) = \frac{t}{s} p_{s-1}(t).
\end{equation}
The first two of them imply that $\{h_s(t)\}_{s \geq 0}$ is a bounded sequence.

\begin{corollary}\label{cor:h-bnd}
For any $s \in \NN_0$ and $t > 0$, we have
\begin{equation}\label{h-bnd}
0 < h_s(t) < 1.
\end{equation}
\end{corollary}

\begin{proof}
Obviously, $p_s(t)$ is positive, therefore $h_s(t) > 0$ is equivalent to $F_{s+1}(t)>0$. Similarly, $h_s(t) < 1$ is equivalent to $F_s(t)> 0$, see \eqref{ids-aFh}. The rest follows from  Lemma \ref{lem:F-pos}.
\end{proof}

The following three lemmas deal with the proof of \eqref{errTaylor-h} for any real $t>0$ and $s \in \NN$. From \eqref{ids-aFh} we observe
\begin{equation}
h_s = 1 - \frac{F_s}{p_s} = 1 - \frac{s}{t}\frac{F_s}{p_{s-1}} = 1 -\frac{s}{t} h_{s-1},
\end{equation}
hence the sequence $\{h_s\}_{s=0}^\infty$ is defined via the recurrence relation
\begin{equation}\label{hs-recur}\left\{
\begin{split}
h_s(t) & = 1 - \frac{s}{t} h_{s-1}(t), \qquad s \in \NN, \\
h_0(t) & = 1 - e^{-t}.
\end{split} \right.
\end{equation}

\begin{lemma}\label{lem:s-small} The two-sided bound 
\begin{equation}\label{s-small}
\frac{t}{t+s+1} < h_s(t) < \frac{t}{s+t}                                                      
\end{equation}
holds for every $t>0$ and $s \in \NN_0$ if $s \leq t-1$.
\end{lemma}

\begin{proof}
The proof is by induction on $s$. By simple calculations we have
\begin{equation}
 \frac{t}{t+1} < h_0(t) = 1-e^{-t} < 1
\end{equation}
and the basis is true. To prove the inductive step, we show that \eqref{errTaylor-h} implies \eqref{s-small}
as long as $s \leq t-1$.
The upper bound follows directly from the lower bound in \eqref{errTaylor-h} and \eqref{hs-recur}, precisely
\begin{equation}
 h_s = 1 - \frac{s}{t} h_{s-1} < 1 -\frac{s}{t} \frac{t}{t+s} = 1 -\frac{s}{t+s} =  \frac{t}{t+s}.
\end{equation}
For the lower bound we have 
\begin{equation}
  h_s = 1 - \frac{s}{t} h_{s-1} > 1 -\frac{s}{t}  \frac{t}{t+s-1} = 1 -\frac{s}{t+s-1} =\frac{t-1}{t+s-1}
\geq \frac{t}{t+s+1}
\end{equation}
where the last estimate holds if and only if $s \leq t-1$. Indeed, it is equivalent to
\begin{equation}
\begin{split}
(t-1)(t+s+1) &\geq t(t+s-1) \\
\Leftrightarrow \quad -(t+s) + t-1 &\geq -t\\
\Leftrightarrow \quad t-1 &\geq s.
\end{split}
\end{equation}
\end{proof}

\begin{lemma}\label{lem:h-lower}
For every $t>0$ and $s \in \NN_0$, we have
\begin{equation}\label{h-lower}
 \frac{t}{t+s+1} < h_s(t).
\end{equation}
\end{lemma}

\begin{proof}
First we observe that \eqref{h-lower} holds true if and only if the sequence $\{h_s\}_{s \geq 0}$ is strictly decreasing. To prove this statement we define the increments
\begin{equation}\label{def-Delta}
 \Delta_s(t) := h_{s+1}(t) - h_s(t) = 1 - \frac{t+s+1}{t} h_s(t), \qquad s \in \NN_0,
\end{equation}
where the last relation follows from \eqref{hs-recur}. Then \eqref{h-lower} is equivalent to
\begin{equation}\label{Delta-neg}
 \Delta_s(t) < 0, \qquad \forall s \in \NN_0.
\end{equation}
Lemma \ref{lem:s-small} implies \eqref{h-lower} and hence  \eqref{Delta-neg} for $s \leq t-1$. Suppose now that $s>t-1$ and show \eqref{Delta-neg} by contradiction. For this, we utilize \eqref{hs-recur} to obtain the following recurrence relation for $\Delta_s$:
\begin{equation} \label{eq[Delta_s]}
 \Delta_s = \frac{s-t+1}{t(s+t-1)}  
+ \frac{s(s-1)}{t^2} \frac{t+s+1}{t+s-1}\Delta_{s-2}.
\end{equation}
Indeed, we have
\begin{equation}
 \Delta_{s-1} = h_s - h_{s-1} = h_s - \frac{t}{s} (1 - h_s) = - \frac{t}{s} + \frac{t+s}{s} h_s, 
\end{equation}
and therefore,
\begin{equation}
h_s = \frac{s}{t+s} \left(\Delta_{s-1} + \frac{t}{s} \right).
\end{equation}
Hence, by \eqref{def-Delta}
\begin{equation}
\begin{split}
 \Delta_s &= 1 - \frac{t+s+1}{t} \frac{s}{t+s} \left(\Delta_{s-1} + \frac{t}{s} \right)
= 1 - \frac{t+s+1}{t+s} - \frac{s(t+s+1)}{t(t+s)} \Delta_{s-1}\\
&=- \frac{1}{t+s} - \frac{s(t+s+1)}{t(t+s)} \Delta_{s-1}\\
&=- \frac{1}{t+s} + \frac{s(t+s+1)}{t(t+s)} 
\left( \frac{1}{t+s-1} + \frac{s-1}{t} \frac{t+s}{t+s-1} \Delta_{s-2} \right)\\
&= \frac{1}{t+s}  \left( -1 + \frac{s(t+s+1)}{t(t+s-1)} \right)
+ \frac{s(s-1)}{t^2} \frac{t+s+1}{t+s-1}\Delta_{s-2}\\
&= \frac{s^2+s-t^2+t}{(t+s)t(t+s-1)}  
+ \frac{s(s-1)}{t^2} \frac{t+s+1}{t+s-1}\Delta_{s-2}\\
&= \frac{s-t+1}{t(s+t-1)}  
+ \frac{s(s-1)}{t^2} \frac{t+s+1}{t+s-1}\Delta_{s-2}.
\end{split}
\end{equation}
Suppose now that \eqref{Delta-neg} is wrong, i.e. $\Delta_{s-2}(t)\geq 0$ for some $s > t+1$. Then from 
\eqref{eq[Delta_s]} it follows that $\Delta_s(t)$ is positive:
\begin{equation}
\Delta_s(t) > \frac{1}{t} \min\left(1,\frac{1}{t}\right) +  \left(1+\frac{1}{t} \right)^2 \Delta_{s-2}> \frac{1}{t} \min\left(1,\frac{1}{t}\right)
+ \Delta_{s-2} > 0.
\end{equation}
Hence, for every $k \in \NN_0$ the increments $\Delta_{s+2k}(t)$ are positive and admit the lower bounds
\begin{equation}
 \Delta_{s+2k}(t) > \frac{1}{t} \min\left(1,\frac{1}{t}\right) + \Delta_{s+2k-2}(t) > \frac{k}{t} \min\left(1,\frac{1}{t}\right)
\to \infty, \qquad k \to \infty.
\end{equation}
This is a contradiction to Corollary \ref{cor:h-bnd} which yields
\begin{equation}
 \Delta_{s+2k}(t) = h_{s+2k+1}(t) - h_{s+2k}(t) < 1.
\end{equation}
The proof is complete.
\end{proof}

\begin{lemma}\label{lem:h-upper}
For every $t>0$ and $s \in \NN_0$, we have
\begin{equation}\label{h-upper}
 h_s(t) < \frac{t}{t+s}.
\end{equation}
\end{lemma}

\begin{proof}
By Lemma \ref{lem:s-small} it suffices to prove \eqref{h-upper} for $s > t-1$. From Lemma \ref{lem:h-lower} we have
\begin{equation}
 \varepsilon_s(t) := h_s(t) - \frac{t}{t+s+1} > 0.
\end{equation}
On the other hand, \eqref{hs-recur} implies
\begin{equation}
\begin{split}
\varepsilon_s(t) + \frac{t}{t+s+1} 
&= h_s(t) 
= 1 - \frac{s}{t} h_{s-1}(t) \\
&= 1 - \frac{s}{t+s} - \frac{s}{t} \varepsilon_{s-1}(t) 
= \frac{t}{t+s} - \frac{s}{t} \varepsilon_{s-1}(t).
\end{split}
\end{equation}
This means that
\begin{equation}
 \varepsilon_s(t) = \frac{t}{(t+s)(t+s+1)} - \frac{s}{t} \varepsilon_{s-1}(t) > 0.
\end{equation}
Hence,
\begin{equation}
\begin{split}
 h_{s-1}(t) &= \varepsilon_{s-1}(t) + \frac{t}{t+s}<  
\frac{t^2}{s(t+s)(t+s+1)} + \frac{t}{t+s}\\
&=\frac{t(t + st + s^2 + s)}{s(t+s)(t+s+1)}
= \frac{t(s+1)}{s(t+s+1)} < \frac{t}{t+s-1},
\end{split}
\end{equation}
where the last inequality holds if and only if $s > t-1$. 
Changing $s-1 \to s$ yields \eqref{h-upper} for $s> t-2$. The proof of \eqref{errTaylor-h} is complete.
\end{proof}

\section{Upper estimates and tractabilities} 
\label{Tractabilities}

~\newline

In this section, we establish some upper bounds for 
$d_N$ and $n_\varepsilon$ defined in \eqref{dNne-def}
and  apply them in tractability studies. 
As auxiliary results for this and the next sections we present upper bounds of the error estimates of the HC approximation by trigonometric polynomials with frequencies from the HC $\Gamma_\pm(s,T,a)$ as well Bernstein inequality for trigonometric polynomials with frequencies from the HC $\Gamma_\pm(s,T,a)$.

\subsection{HC approximations} 
\label{Hyperbolic cross approximations}

~\newline
For a finite set $M \subset \ZZs$, we  denote by  the $\mathcal T(M)$
subspace of trigonometric polynomials with frequencies in $M$, i.e., trigonometric polynomials $g$ of the form 
\begin{equation}
g = \sum_{\bk \in M} \hat{g}(\bk) e_{\bk}.
\end{equation}
We abbreviate $\mathcal T(s,T,a) := \mathcal T(\Gamma_\pm(s,T,a))$. 

For a function $f \in L_2(\TTs)$, we define the  Fourier operator $S_T$ as
\begin{equation}
S_T(f)
:= \
\sum_{\bk \in \Gamma_\pm(s,T,a)}  \hat{f}(\bk) e_{\bk}.
\end{equation}
Obviously, $S_T$  is the orthogonal projection onto $\mathcal T(s,T,a)$. 

We will need the following equation which is derived from the definition of the norm in $\tKrs$ and Parseval's identity
\begin{equation} \label{[Parseval-Id]}
\|f\|_{\tKrs}^2 
\ = \
\sum_{\bk \in \ZZ^s} |\lambda_a(\bk)|^{2r}|\hat{f}(\bk)|^2. 
\end{equation}
The following lemma and corollary give upper bounds with respect to $T$ for the error of the orthogonal projection.

\begin{lemma} \label{lem:Jackson-per}
For arbitrary $T \geq 1$, we have
\begin{equation}
\|f - S_T(f)\|
\ \le  \
T^{-r}\|f\|_{\tKrs}\, , \qquad  \forall f \in \tKrs.
\end{equation}
\end{lemma}

\begin{proof}
Indeed,  from \eqref{[Parseval-Id]} we have for every $f \in \tKrs$,
\begin{equation*} 
\begin{aligned}
\|f - S_T(f)\|^2
&= \  
\sum_{\bk \not\in \Gamma(s,T)}  \|\hat{f}(\bk)\|^2 \\
& \le \ 
\sup_{\bk \not\in \Gamma(s,T)}\big[\lambda_a(\bk)^{-2r}\big]
\ 
\sum_{\bk \not\in \Gamma(s,T)} \big[\lambda_a(\bk)^{2r}\big]\| \hat{f}(\bk)\|^2 \\
 &\le \
T^{-2r} \|f\|_{\tKrs}^2.
\end{aligned}
\end{equation*}
\end{proof}

\begin{corollary} \label{corollary[|f - S_xi(f)|]}
For arbitrary $T \geq 1$, we have
\begin{equation}
\sup_{f \in \tUr} \ \inf_{g \in \mathcal T(s,T,a)} \|f - g\|
\ =  \
\sup_{f \in \tUr} \|f - S_T(f)\|
\ \le  \
T^{-r}, \\ 
\end{equation}
\end{corollary}


Next, we prove a Bernstein type inequality. 

\begin{lemma} \label{lem:Bernstein-per}
For arbitrary $T \ge 1$, we have
\begin{equation}
\|f\|_{\tKrs}
\ \le  \
T^r \|f\|, \ \  f \in \mathcal T(s,T,a).
\end{equation}
\end{lemma}

\begin{proof}
Indeed,  by \eqref{[Parseval-Id]} we have for every $f \in \mathcal T(s,T,a)$,
\begin{equation*} 
\begin{aligned}
\|f\|_{\tKrs}^2
\ &= \ 
\sum_{k \in \mathcal T(s,T,a)}  \big[\lambda_a(\bk)^{2r}\big]|\hat{f}(\bk)|^2 \\
\ & \le \
\sup_{k \in \mathcal T(s,T,a)} \big[\lambda_a(\bk)^{2r}\big] \ 
\sum_{k \in \mathcal T(s,T,a)} |\hat{f}(\bk)|^2  
 \le \
T^{2r}\|f\|^2.
\end{aligned}
\end{equation*}
\end{proof}

\begin{corollary} \label{corollary[<d_N<]}
Let $T \ge 1$ and $N= |\Gamma_\pm(s,T,a)|$. Then we have 
\begin{equation} \label{ineq[<d_N<]}
d_N
\ \le  \
T^{-r}
\ \le  \
d_{N-1}.
\end{equation}
\end{corollary}

\begin{proof} 
The first inequality in \eqref{ineq[<d_N<]} follows from  Corollary \ref{corollary[|f - S_xi(f)|]} and the equation
\begin{equation} \label{[dimT(s,T,a)]}
 \dim \mathcal T(s,T,a)  \ = \ N.   
\end{equation}
To establish the second one, we need the following result proven by Tikhomirov \cite[Theorem 1]{Ti60}. 
Let $L_n$ be an $n$-dimensional subspace in a Banach space $X$, and 
$B_n(\delta):= \{f \in L_n: \ \|f\|_X \le \delta \}$. Then,
\begin{equation} \label{[d_{N-1}(B_N,X)]}
d_{n-1}(B_n(\delta), X)
\ = \
\delta.
\end{equation}
Consider the subset 
$
B(T) := \{f \in \mathcal T(s,T,a): \ \|f\| \le T^{-r}\}
$
in $ L_2(\TTs)$. 
By Lemma \ref{lem:Bernstein-per} we have $B(T) \subset \tUr$. 
Applying \eqref{[d_{N-1}(B_N,X)]}, by \eqref{[dimT(s,T,a)]} we get
 \begin{equation} \nonumber
d_{N-1} 
\ \ge \
d_{N-1}(B(T), L_2(\TTs))
\ =  T^{-r}.
\end{equation}
\end{proof}

From Corollary \ref{corollary[<d_N<]} we derive 

\begin{corollary} \label{corollary[<n_e<]}
Let $\varepsilon \in (0,1]$. Then we have 
\begin{equation} \label{ineq[<n_e<]}
|\Gamma_\pm(s,\varepsilon^{-1/r},a)| - 1
\ \le  \
n_\varepsilon
\ \le  \
|\Gamma_\pm(s,\varepsilon^{-1/r},a)|
\end{equation}
\end{corollary}

 \subsection{Upper bounds of $N$-widths and $\varepsilon$-dimensions}
 ~\newline
\begin{theorem} \label{theorem[d_N<(1)]}
Let $r > 0$, $s \in \IN$, $a > 0$. Then for any $q \in [2,\infty)$ satisfying  the inequality  
$\lambda:= a - 2/q > 0$, and any $N \in \NN$, we have 
\begin{equation} \label{[d_N<(1)]}
d_N 
\  \leq \ 
2^rq^{r/(1 + q)}\lambda^{-[qr/(1+q)]s}N^{-r/(1+q)}.
\end{equation}
\end{theorem}

\begin{proof}
Put $N(T) := |\Gamma_\pm(s,T,a)|$. Then $N(\cdot)$ is an increasing step function in the
variable $T > 0$. Note that $N(T)= 0$ if $0 < T < a^s$, and $N(a^s)= 1$. Hence, there are strictly increasing sequences of positive numbers $\{T_m\}_{m=0}^\infty$ and $\{N_m\}_{m=0}^\infty$ such that 
$T_0 =a^s$ $N_0= 1$ and the function $N(\cdot)$ is given by the formula
\begin{equation}
N(T) 
\ = \
N_m, \ \ T_m \le T < T_{m+1}, \quad m \in \NN_0.
\end{equation}
Notice that 
\begin{equation} \label{ineq[T_{m+1}-T_m(1)]}
 T_{m+1}
\ \le \
2T_m.
\end{equation}
Indeed, let 
\begin{equation} 
T_m =  \prod_{j=1}^s (|k_j| + a)
\end{equation} 
 for some $\bk \in \Gamma_\pm(s,T,a)$.  Define $\bk' \in \NNs$
by $k'_s = k_s + 1$ and $k'_j = k_j, \ j \not= s$. Then we have
\begin{equation}
\begin{aligned}
 T_{m+1} - T_m 
\ &\le \
\prod_{j=1}^s (|k'_j| + a) - \prod_{j=1}^s (|k_j| + a) \\
\ &= \
\prod_{j=1}^{s-1} (|k_j| + a)
\ \le \
T_m.
\end{aligned}
\end{equation}

Put $\delta = 2/q$. 
By Corollary \ref{corollary[Gamma<(3)]} we have for every $\delta \in (0,1]$ with $a - \delta >0$, 
\begin{equation}
 N_m \ < \  (2/\delta) T_m^{1 + 2/\delta} (a - \delta)^{-2s/\delta}.
\end{equation}
Hence,
\begin{equation} \label{ineq[T_m^{-1}(1)]}
  T_m^{-1} 
\ < \ 
\left[(2/\delta)(a - \delta)^{-2s/\delta}N_m^{-1}\right]^{1/(1 + 2/\delta)} .
\end{equation}

Let $N \in \NN_0$ be an arbitrary number. There is a $m \in \NN$ such that 
$N_{m-1} \le N < N_m$. By  Corollary \ref{corollary[<d_N<]}, the inequalities  
\eqref{ineq[T_{m+1}-T_m(1)]} and \eqref{ineq[T_m^{-1}(1)]},
\begin{equation}
\begin{aligned}
 d_N 
\ &\le \
d_{N_{m-1}}  \\
\ &\le \
T_{m-1}^{-r} \\
\ &\le \ \big[2 T_m^{-1} \big]^r \\
\ &\le \
2^r\left[(2/\delta) (a - \delta)^{-2s/\delta}N_m^{-1}\right]^{r/(1 + 2/\delta)} \\
\ &\le \
2^r\left[(2/\delta)(a - \delta)^{-2s/\delta}N^{-1}\right]^{r/(1 + 2/\delta)} \\
\ &= \ 
2^rq^{r/(1 + q)}\lambda^{-[qr/(1+q)]s}N^{-r/(1+q)}.
\end{aligned}
\end{equation}
\end{proof}

{We remark that if the $a>1$ 
in Theorem  \ref{theorem[d_N<(1)]}, then we can choose $q \in [2,\infty)$  so that $\lambda:= a - 2/q > 1$ and then the upper bound \eqref{[d_N<(1)]} implies that the $N$-widths 
$d_N$ decreases exponentially if the the dimension $s$ grows}. This situation corresponds to the exponential tractability of  the problem of $n_\varepsilon$ for the case $a>1$ presented in Theorem 
\ref{tractability}(i) in the next subsection.

\begin{theorem} \label{theorem[n_varepsilon]}
For any $q \in [2,\infty)$ satisfying  the inequality  
$\lambda:= a - 2/q > 0$ and any $\varepsilon \in (0,1]$, we have  
\begin{equation} \label{[n_varepsilon<]}
n_\varepsilon 
\  \leq \ 
q \lambda^{-qs}\varepsilon^{-(1+q)/r}.
\end{equation}
\end{theorem}

\begin{proof} Put $\delta = 2/q$ a given $q \in [2,\infty)$ with $\lambda:= a - 2/q > 0$, and  
$T = \varepsilon^{-1/r}$ for a given $\varepsilon \in (0,1]$.
 By Corollaries \ref{corollary[<n_e<]} and  \ref{corollary[Gamma<(3)]} we have 
\begin{equation} \label{[n_varepsilon<(0)]}
\begin{aligned}
 n_\varepsilon 
\ &\le \
 |\Gamma_\pm(s,T,a)|   \\
\ &  < \
2\delta^{-1} T^{1 + 2/\delta} (a - \delta)^{-2s/\delta} 
\ = \ 
q \lambda^{-qs}\varepsilon^{-(1+q)/r}.
\end{aligned}
\end{equation}
\end{proof}

\subsection{Tractabilities of the problem of $\varepsilon$-dimensions}
~\newline
Usually, in Information-Based Complexity, we assume that the
information is linear (like Fourier coefficients) and this leads to the Gel'fand $N$-widths $d^N(W,X)$. Often they are smaller than the linear $N$-widths 
$\lambda_N(W,X)$. In a Hilbert space $X$, the Kolmogorov $N$-widths $d_N(W,X)$ and linear $N$-widths $\lambda_N(W,X)$ are the same, but the Gel'fand $N$-widths $d^N(W,X)$ can be smaller. However, as mentioned in Introduction,
the equations \eqref{eq[three-widths]} together with Corollary \ref{corollary[|f - S_xi(f)|]} on the Fourier approximation allow us to study the tractability of the problem of $n_\varepsilon$ in the  sense information complexity for the corresponding traditional linear problem in the worst case setting, see \cite[Section 4.2, Chapter 4]{NW08} for details. 

Let us give notions of tractability for a problem of $\varepsilon$-dimensions. Let $W$ be a subset in $L_2(\TTs)$. We says that the problem of $\varepsilon$-dimensions of $W$ is \emph{weakly tractable} if 
\begin{equation} \label{def[weak-tractability]}
\lim_{\varepsilon^{-1} + s \ \to \ \infty} \frac{\ln n_\varepsilon (W,L_2(\TTs))}{\varepsilon^{-1} + s} 
\ = \ 0,
\end{equation}
and \emph{intractable} if this equation does not hold. 
 The problem of $\varepsilon$-dimensions of $W$ is \emph{polynomially tractable} if there are nonnegative numbers $C$, $p$ and $q$ such that
\begin{equation} \label{def[poly-tractable]}
 n_\varepsilon (W,L_2(\TTs)) 
\ \le \ C  s^q \varepsilon^{-p}\quad  \text{for all} \ \varepsilon \in (0,1] \ \text{and} \ s \in \NN.
\end{equation} 
The problem is {\em strongly polynomially tractable} if  $q=0$. For details about notions of tractability see 
\cite[Section 4.4, Chapter 4]{NW08}.

Let us introduce a new notion of tractability. Namely, the problem of $\varepsilon$-dimensions of $W$ is called \emph{exponentially tractable} if there are nonnegative numbers $C$, $p$ and a positive number $q$ such that
\begin{equation} \label{def[exp-tractability]}
 n_\varepsilon (W,L_2(\TTs)) 
\ \le \ C  e^{-q s} \varepsilon^{-p} \quad  \text{for all} \ \varepsilon \in (0,1] \ \text{and} \ s \in \NN.
\end{equation}
When the problem is exponentially tractable {\em the exponent of exponential tractability}  is called the quantity $p^{\exp}$ defined as the infimum of $p$ for which there holds \eqref{def[exp-tractability]} for some $C$ and $q>0$.  An exponentially tractable problem is strongly polynomially tractable.

The following theorem describes the tractability of the problem of $\varepsilon$-dimensions of the class $\tUr$. 

\begin{theorem} \label{tractability}
Let $r >  0$, $s \in \IN$, $a > 0$. Then there holds the following.
\begin{itemize}
\item[(i)]
For $a > 1$, the problem of $n_\varepsilon$ is exponentially tractable. 
 Moreover, for any $q \in [2,\infty)$ satisfying  the inequality  
$\lambda:= a - 2/q > 1$, and any $\varepsilon \in (0,1]$ we have  
\begin{equation}
n_\varepsilon 
\  \leq \ 
q \lambda^{-qs}\varepsilon^{-(1+q)/r},
\end{equation}
and 
\begin{equation}
p^{\exp} 
\  \leq \ 
\begin{cases}
3/r, \ & a \ge 2 \\[1.5ex]
(1+ 2/(a-1))/r, \ & a < 2. 
\end{cases}
\end{equation}
\item[(ii)]
For $a = 1$, the problem of $n_\varepsilon$ is weakly tractable but polynomially intractable. 
\item[(iii)]
For $a < 1$, the problem of $n_\varepsilon$ is intractable. 
\end{itemize}
\end{theorem}

\begin{proof} Assertion (i) follows from Theorem \ref{theorem[n_varepsilon]} with $\delta = 2/q$. 

To prove Assertions (ii) and (iii) resort 
to the question of tractability of a linear tensor product problem over the Hilbert spaces $\tKrs$ and $L_2(\TTs)$ 
(see \cite[Section 5.1, Chapter 5]{NW08}) for details) and then applying  \cite[Theorem 5.5]{NW08} on the tractability of a linear tensor product problem.  
We give a direct proof by applying the results of $L_2(\TTs)$-approximations of functions in $\tKrs$, and on the upper and lower bounds of $|\Gamma(s,T,a)|$ obtained in the present paper. 

For a given $\varepsilon \in (0,1]$, by Corollary \ref{corollary[<n_e<]} and the inequality 
$|\Gamma_\pm(s,T,1)| - 1 > |\Gamma(s,T,1)|$ the it holds that
\begin{equation} \label{[<n_varepsilon<]}
|\Gamma(s,T,1)|
\ < \
 n_\varepsilon 
\ \le \
|\Gamma_\pm(s,T,1)|, \quad T= \varepsilon^{-1/r}.
\end{equation}
The inequalities \eqref{[<n_varepsilon<]}  tell us that the problem of tractability or intractability for $n_\varepsilon$ can be reduced to estimations from above of $|\Gamma_\pm(s,T,1)|$ or from bellow of $|\Gamma(s,T,1)|$, respectively. 

\medskip
\noindent
Assertion (ii). We first show that the problem of $n_\varepsilon$ with $a=1$ is weakly tractable.
Let us estimate $|\Gamma_\pm(s,T,1)|$.
 By Lemma \ref{lem:PQ-incl} and Corollary \ref{cor:I-bnd} it holds that
\begin{equation}
\begin{split}
|\Gamma_\pm(s,T,1)| & < \sum_{k=0}^s \binom{s}{k} 2^k|\Gamma(k,T,2)|
<\sum_{k=0}^s \binom{s}{k} 2^k I(k,T,1) \\
&<\sum_{k=0}^s \binom{s}{k} 2^k \frac{T (\ln T)^k}{(k-1)!(\ln T + k-1)}\\
&<T\sum_{k=0}^s \binom{s}{k} \frac{(2\ln T)^k}{k!}
<T\sum_{k=0}^s \frac{(s2\ln T)^k}{(k!)^2}.
\end{split}
\end{equation}
By Stirling's approximation $k! \geq \left(\frac{k}{e}\right)^k$ we have
\begin{equation}
T\sum_{k=0}^s \frac{(s\ln T)^k}{(k!)^2}
\leq T \sum_{k=0}^s \left( \frac{s2\ln T e^2}{k^2}\right)^k.
\end{equation}
The function $ f(x) = \left(\frac{a}{x^2}\right)^x$ has a unique maximum at the positive semi-axis, attained at 
$x_* = \frac{\sqrt{a}}{e}$. We have $f(x_*) = \exp(2\sqrt{a}/e)$ Therefore, 
\begin{equation}
T \sum_{k=0}^s \left( \frac{s2\ln T e^2}{k^2}\right)^k
\leq Ts \exp(2 \sqrt{s 2\ln T})
\end{equation}
By H\"older inequality, e.g. with $p=\frac{3}{2}$ and $q = 3$ we get
\begin{equation}
 \sqrt{s 2\ln T} \leq \frac{2}{3}s^{\frac{3}{4}} + \frac{1}{3}(2\ln T)^{\frac{3}{2}}
\end{equation}
and therefore, 
\begin{equation}
\begin{split}
 |\Gamma_\pm(s,T,1)| 
\ &\leq \ 
T s \exp\left(\frac{4}{3} s^{\frac{3}{4}}\right) \exp\left( \frac{2}{3} (2\ln T)^{\frac{3}{2}}\right) \\[1.5ex]
\ &= \ 
\varepsilon^{-1/r} s \exp\left(\frac{4}{3} s^{\frac{3}{4}}\right) 
\exp\left( \frac{2}{3} [2\ln (\varepsilon^{-1/r})]^{\frac{3}{2}}\right).
\end{split}
\end{equation}
This together with the second inequality in \eqref{[<n_varepsilon<]} proves the weak tractability of the problem of $n_\varepsilon$ with $a=1$. 

Next we show that the problem of $n_\varepsilon$ with $a=1$ is polynomially intractable. To this end we will find a sequence $\varepsilon_m, s_m$ such that for every fixed triple $C,p,q >0$ there exists an index $m_*$ such that for all $m \ge m^*$,
\begin{equation}\label{aim-nonpoly}
 n_{\varepsilon_m}(U^{r,a}_{s_m},L(\TT^{s_m})) \geq C \varepsilon_m^{-p} s^q_m
\end{equation}
and therefore \eqref{def[poly-tractable]} does not hold.
For this, let us consider a sequence
\begin{equation}
 s_m = m^2, \quad T_m = 2^m, \quad \varepsilon_m := T_m^{-r} = 2^{-rm}, \qquad m \in \IN.
\end{equation}
We have by the first inequality in \eqref{[<n_varepsilon<]},
\begin{equation} \label{[Gamma>m^m]}
 n_{\varepsilon_m}(U^{r,a}_{s_m},L(\TT^{s_m})) > |\Gamma(s_m,T_m,1)|
= \sum_{k=0}^{s_m} \binom{s_m}{k}|\Gamma(k,T_m,2)| \geq 
\binom{m^2}{m} > m^m.
\end{equation}
Obviously, the function $m^m$ grows more rapidly than $C \varepsilon_{m}^{-p} s^q_{m} = C 2^{rpm} m^{2q}$ for any fixed triple $C,p,q$. Therefore, there exists a thresholding value $m_* = m_*(C,p,q)$  such that
\begin{equation}
C \varepsilon_m^{-p} s_m^q = C 2^{rpm} m^{2q} < m^m, \qquad \forall m \geq m_*, 
\end{equation}
and \eqref{aim-nonpoly} holds true because of \eqref{[Gamma>m^m]}.

\medskip
\noindent
Assertion (ii). Finally let us prove  the problem of $n_\varepsilon$ with $a < 1$ is intractable. 
Precisely, we will find a sequence $\varepsilon_m,s_m$ satisfying 
$\varepsilon_m^{-1}+s_m \to \infty$, $m \to \infty$, such that
\begin{equation} \label{ineq[intractability]}
 \liminf_{m \to \infty} \frac{\ln n_{\varepsilon_m}(U^{r,a}_{s_m},L(\TT^{s_m}))}{\varepsilon_m^{-1}+s_m}
\geq C' > 0.
\end{equation}
We claim that 
\begin{equation}
 s_m := m, \quad T_m := 1\quad \varepsilon_m := 1 \qquad \mbox{ for } m \in \NN
\end{equation}
is the sought sequence. Indeed, $\varepsilon_m^{-1}+s_m = 1 + m \to \infty$ as $m \to \infty$. For a fixed pair of parameters $(a,m)$ we denote
\begin{equation}
 k_m := \max\{ k \in \NN: (1+a)^k a^{m-k} \leq 1\} = \lfloor m A \rfloor, \qquad A = \frac{\ln \tfrac{1}{a}}{\ln(1+a) +\ln\tfrac{1}{a} } > 0.
\end{equation}
Obviously, $m$ denotes the number of dimensions in which $\Gamma(m,1,a)$ contains at least two elements.
Then from Stirling's approximation it holds that
\begin{equation}
|\Gamma(m,1,a)| \geq \binom{m}{k_m} 2^{k_m} \geq 
 2^{m A-1}.
\end{equation}
This together with the first inequality in \eqref{[<n_varepsilon<]} implies that
\begin{equation}
 \frac{\ln n_{\varepsilon_m}(U^{r,a}_{s_m},L(\TT^{s_m}))}{\varepsilon_m^{-1}+s_m}
>  \frac{\ln \Gamma(m,1,a)}{1+m}
\ge \frac{(mA-1) \ln 2}{1+m} \to A \ln 2 > 0, \ m \to \infty,
\end{equation}
and consequently, \eqref{ineq[intractability]}.

The theorem is completely proven.
\end{proof}

\begin{remark}
Let $a_0 \in (0,1)$ be the unique solution of the equation $a(a+1)^{2r}-1 = 0$ on the interval $(0,1]$. Then we have that for any $0<a \le a_0$ and 
$\varepsilon \in (0,1]$, 
\begin{equation} 
n_\varepsilon 
\  \geq \ 
2^{s-1},
\end{equation}
and for any $a_0 < a < 1$ and $\varepsilon \in (0,1]$,
\begin{equation} 
n_\varepsilon 
\  \geq \ 
\exp \{[\alpha \ln \alpha^{-1} + (1-\alpha) \ln (1-\alpha)^{-1}]s - \ln s + O(1) \}, \ s \to \infty,
\end{equation}
 where 
\begin{equation} 
\alpha := \frac{\ln a^{-1}}{\ln [a^{-1}(a+1)^{2r}]}.
\end{equation}
\end{remark}

The tractability of a problem can be studied in two different aspects: for the absolute error criterion and for the normalized error criterion, see \cite[Section 4.4, Chapter 4]{NW08} for definitions and necessary facts. The tractability described in Theorem \ref{tractability} is for the absolute error criterion. Let us consider the tractability problem for the normalized error criterion. To this end, let us consider the set 
\[
\tUrn := a^{rs}\, \tUr = \{ f \in \tKrs: \|f\|_{\tKrs} \le a^{rs} \}.
\]
Note that due to the relations
\[
d_N(\tUrn, L_2(\TTs)) 
\ = \ 
a^{rs}d_N 
\]
and
 \[
n_\varepsilon(\tUrn, L_2(\TTs)) 
\ = \ 
n_{a^s\varepsilon},
\]
the problem of lower and upper estimations of $d_N(\tUrn, L_2(\TTs))$ and $n_\varepsilon(\tUrn, L_2(\TTs))$ is reduced to the problem of lower and upper estimations of $d_N$ and $n_\varepsilon$.

Since we have 
\begin{equation}
\sup_{f \in \tUrn} \|f\| = 1,
\end{equation}
the tractability of the linear tensor product problem of $n_\varepsilon(\tUrn, L_2(\TTs))$ is for the normalized error criterion. From \cite[Theorem 5.5]{NW08} we obtain the following

\begin{theorem} \label{tractability[normalized error]}
Let $r >  0$, $s \in \IN$, $a > 0$. Then the problem of 
$n_\varepsilon(\tUrn, L_2(\TTs))$ is weakly tractable but polynomially intractable.
\end{theorem}

\section{Sharpened estimates of $N$-widths and $\varepsilon$-dimensions} \label{widths and dimensions}

~\newline
In this section, we sharpen the results of the previous section by 
establishing quantitative upper and lower bounds for the quantities 
$d_N$ and $n_\varepsilon$ as a function of three variables 
$N,s,a$ and  $\varepsilon,s,a$, respectively.  To this end we will need some estimates of the inverse of  $|\Gamma_\pm(s,T,a)|$ as well $|\Gamma(s,T,a)|$ for the next section as a function of the variable $T$.

\subsection{A preliminary: inverted estimates} \label{sec:inverted}
~\newline
Lemma \ref{lem:PQ-incl} and Corollaries \ref{cor:I-bnd} and \ref{cor:IGamma-newbnd} provide the upper and lower bounds for $\delta \in [\tfrac{1}{2},1]$, $a-\delta > 0$ and $N := |\Gamma(s,T,a)|$,
\begin{equation}\label{up-lo-bnd}
\frac{T (\ln T - s\ln a)^s }{(s-1)!(\ln T + s(1-\ln a))}
 < 
N 
 < 
\frac{T \big(\ln T - s\ln (a-\delta)\big)^s }{(s-1)!\big(\ln T + s\big(1-\ln (a - \delta)\big) - 1\big)}
\end{equation}
if $T > T_*(s,\delta)$. 
In the particular case $\delta = 1$, it holds $T_*(s,1) = a^s$. For the upcoming estimation of $N$-widths, we will require an inverse to \eqref{up-lo-bnd} estimate, namely upper and lower bounds for $T^{-1}$ in terms of the cardinality $N$.

\begin{theorem}\label{thm:inv-bnd}
Suppose $s \in \IN$, $s \geq 2$, $\delta \in [\tfrac{1}{2},1]$ and $a - \delta > 0$ are fixed. Then for $N := |\Gamma(s,T,a)|$, it holds that
\begin{equation}\label{inv-bnd}
\begin{split}
 \frac{2}{s+2} \bigg[N(s-1)!\bigg]^{-1} \left[ \ln \frac{N(s-1)!}{a^s} - (s-1) \ln \ln \frac{N(s-1)!}{(a-\delta)^s} \right]^{s-1}\\
\ < \ \frac{1}{T} \ < \
\bigg[N(s-1)!\bigg]^{-1} \left[ \ln \frac{N(s-1)!}{(a-\delta)^s} \right]^{s-1}
\end{split}
\end{equation}
if $T$ is large enough. In particular, $T \geq a^s e^2$ if $\delta = 1$, and for 
$\delta \in [\tfrac{1}{2},1)$ we require $T \geq \max\big(a^s e^2, T_*(s,a) \big)$ with $T_*(s,a)$ from \eqref{Gamma-asy}.
\end{theorem}

\begin{proof}
To simplify the notations, we denote $u := a^s$, $v := (a-\delta)^s$ and $M := (s-1)!N$. Then the estimate \eqref{up-lo-bnd} reads
\begin{equation}\label{TM-bnd}
T \ln^{s-1} (T/u)< \xi M, 
\qquad 
\eta M < T \ln^{s-1} (T/v)
\end{equation}
for $\xi = \xi(s,T,a,\delta)$, $\eta = \eta(s,T,a,\delta)$ given by
\begin{equation}
\xi \ := \
1 + \frac{s}{\ln T - s \ln a}, \qquad
\eta := 1 + \frac{s-1}{\ln T - s \ln(a-\delta)}.
\end{equation}
We notice that for any $c>0$ and $T >a^s \exp c$ it holds
\begin{equation}
 1 < \eta < \xi < \frac{s+c}{c}.
\end{equation}
This means in particular that $\xi, \eta = \mathcal O(1)$ as $T \to \infty$. Our aim is to invert estimates \eqref{TM-bnd} and to bound $T^{-1}$ in terms of $M$. Logarithmizing \eqref{TM-bnd} we get equivalent bounds
\begin{equation}\label{lnTM-bnd}
\ln T < \ln (\xi M) - (s-1)\ln \ln (T/u), 
\qquad 
 \ln T >  \ln(\eta M) - (s-1)\ln \ln (T/v).
\end{equation}
Using \eqref{TM-bnd} and \eqref{lnTM-bnd} we have
\begin{equation}\label{inv-bnd-1}
\frac{\psi^{s-1}}{\xi M} < \frac{1}{T} < \frac{\varphi^{s-1}}{\eta M}
\end{equation}
where $\varphi$, $\psi$ admit the bounds
\begin{equation}\label{phi-bnd}
\begin{split}
\varphi = \ln(T/v)
&<\ln(\xi M/v) - (s-1)\ln \ln (T/u)\\
&<\ln(\xi M/v) - (s-1)\ln \left[ \ln (\eta M/u) - (s-1) \ln \ln(T/v) \right] \\
&< \dots
\end{split}
\end{equation}
\begin{equation}\label{psi-bnd}
\begin{split}
\psi = \ln(T/u)
&>\ln(\eta M/u) - (s-1)\ln \ln (T/v)\\
&>\ln(\eta M/u) - (s-1)\ln \left[ \ln (\xi M/v) - (s-1) \ln \ln (T/u) \right]\\
&> \dots
\end{split}
\end{equation}
These estimates can be continued further in this fashion and are valid upper and lower bounds if the arguments of logarithms are larger than one. In particular, we can truncate \eqref{phi-bnd} after the first line and obtain
\begin{equation}\label{phi-bnd2}
 \varphi < 
\left\{
\begin{array}{cl}
\ln (\xi M/v), & T\geq u e, \\[1.5ex]
\ln ( M/v), & T\geq u e^2. 
\end{array}
\right.
\end{equation}
If in addition $\xi M \geq v \exp(1)$, we can truncate \eqref{psi-bnd} after the second line and obtain
\begin{equation}\label{psi-bnd2}
 \psi > 
\left\{
\begin{array}{ll}
\ln (\eta M/u) - (s-1) \ln \ln (\xi M /v), & T\geq u e, \\[1.5ex]
\ln (\eta M/u) - (s-1) \ln \ln (M /v), & T\geq u e^2.
\end{array}
\right.
\end{equation}
We remark that the trivial estimate
\begin{equation}
\left.
\begin{split}
\frac{\xi M}{v} > \frac{M}{v} \\
\frac{\eta M}{u} > \frac{M}{u}
\end{split}
\right\}
\geq 
\frac{M}{u} \geq \frac{T}{u} \geq e
\end{equation}
provides that the arguments of all logarithms in \eqref{phi-bnd2}, \eqref{psi-bnd2} are \emph{larger than one}.
The assertion of the theorem follows from \eqref{inv-bnd-1}, \eqref{phi-bnd2} and \eqref{psi-bnd2}.
\end{proof}

Analogously, Lemma \ref{lem:IGamma-sym} and Corollary \ref{cor:I-bnd} provide the upper and lower bounds for $N := |\Gamma_\pm(s,T,a)|$ and $T \ge (a + 1/2)^s$, $a > 1/2$. 
\begin{equation}\label{up-lo-bnd-sym}
\frac{2^sT (\ln T - s\ln (a+1/2))^s }{(s-1)!(\ln T + s(1-\ln (a + 1/2))}
 < 
N 
 < 
\frac{2^sT (\ln T - s\ln (a-1/2))^s }{(s-1)!\big(\ln T + s\big(1-\ln (a-1/2)) - 1\big)}.
\end{equation}

Similarly to the arguments in Theorem \ref{thm:inv-bnd} we obtain  the following estimates for $\Gamma_\pm(s,T)$. 

\begin{theorem}\label{thm:inv-bnd-sym}
Suppose $s \in \IN$, $s \geq 2$ and $a > 1/2$ are fixed and $T \geq (a + 1/2)^s e^2$. Then for 
$N := |\Gamma_\pm(s,T,a)|$, it holds that
\begin{equation}\label{inv-bnd-sym}
\begin{split}
 \frac{2}{s+2} \left[\frac{N(s-1)!}{2^s}\right]^{-1} \left[ \ln \frac{N(s-1)!}{(2a +1)^s} 
- (s-1) \ln \ln \frac{N(s-1)!}{(2a-1)^s} \right]^{s-1}\\
\ < \ \frac{1}{T} \ < \
\left[\frac{N(s-1)!}{2^s}\right]^{-1} \left[ \ln \frac{N(s-1)!}{(2a - 1)^s} \right]^{s-1}.
\end{split}
\end{equation}
\end{theorem}



\subsection{Upper and lower bounds}
~\newline
\begin{theorem} \label{theorem[<d_N<]}
Let $r > t \ge 0$, $s \geq 2$, $a > 1/2$. Then for every $N \in \NN$ satisfying the inequality
\begin{equation}\label{Cond[N]}
N \ge 
\frac{2^s}{(s-1)!}\left[\frac{T^* (\ln T^* - s\ln (a-1/2) )^s}{\ln T^* + s\big(1-\ln (a-1/2)) - 1} \right], 
\quad \mbox{uith} \quad T^* := 16\lceil (a+1/2)^s \rceil,
\end{equation}
we have
\begin{equation}\label{ineq[d_N(2)<]}
d_N 
\  \le \ 
2^r \left(\left[\frac{N(s-1)!}{2^s}\right]^{-1} \left[ \ln \frac{N(s-1)!}{(2a - 1)^s} \right]^{s-1}\right)^r,
\end{equation}
and there is a number $N^*(s,a) \in \NN$ such that for every  $N \ge N^*(s,a)$,
\begin{equation}\label{ineq[d_N(2)>]}
d_N 
\  \ge \ 
2^{-1} \left( \left[\frac{N(s-1)!(s+2)}{2^s}\right]^{-1} \left[ \ln \frac{N(s-1)!}{(2a +1)^s} 
\right]^{s-1}\right)^r
\end{equation}
\end{theorem}

\begin{proof} 
Let $T_m:= T^*m$ for $m \in \NN$. Clearly, $\{T_m\}_{m=1}^\infty$ is a strictly increasing sequence 
and  $T_m \to \infty$ when $m \to \infty$, and 
\begin{equation} \label{ineq[T_{m+1}-T_m]}
 T_{m+1}
\ \le \
2T_m.
\end{equation}
Put $N_m := |\Gamma_\pm(s,T_m,a)| = \dim \mathcal T(s,T_m,a)$. Then one can see that the sequence $\{N_m\}_{m=1}^\infty$ is strictly increasing and 
$N_m \to \infty$ when $m \to \infty$.
We have $T_1= T^*$ and for all $m \ge 1$,
\begin{equation}\label{ineq[T-m>]}
T_m  > (a + 1/2)^s e^2.
\end{equation}

Applying Theorem \ref{thm:inv-bnd-sym}, we derive that for  every $m \ge 1$,
\begin{equation} \label{ineq[T_m^{-1}]}
  T_m^{-1} 
\ < \ 
\left[\frac{N_m(s-1)!}{2^s}\right]^{-1} \left[ \ln \frac{N_m(s-1)!}{(2a - 1)^s} \right]^{s-1}.
\end{equation}

Let $N \in \NN$ be an arbitrary number satisfying \eqref{Cond[N]}. 
There is a $m \in \NN$ such that 
$N_{m-1} \le N < N_m$. Hence, by \eqref{Cond[N]}, Lemma \ref{lem:IGamma-sym} and  Corollary \ref{cor:I-bnd} 
$N_m > N_1$ and consequently, $T_m > T_1 \ge (a + 1/2)^s e^2$. 
By the definition, Corollary 
\ref{corollary[<d_N<]}, the inequalities  \eqref{ineq[T_{m+1}-T_m]},  \eqref{ineq[T_m^{-1}]}, \eqref{Cond[N]} and the decreasing of the function $g(x): = x^{-1} \ln^{s-1}x$ for $x \le e^{-(s-1)}$, we obtain
\begin{equation}
\begin{aligned}
 d_N 
\ &\le \
d_{N_{m-1}}  \\
\ &\le \
T_{m-1}^{-r} \\
\ &\le \ \big[2 T_m^{-1} \big]^r \\
\ &\le \ \left( 2
\left[\frac{N_m(s-1)!}{2^s}\right]^{-1} \left[ \ln \frac{N_m(s-1)!}{(2a - 1)^s} \right]^{s-1}
\right)^r \\
\ &\le \
2^r\left(\left[\frac{N(s-1)!}{2^s}\right]^{-1} \left[ \ln \frac{N(s-1)!}{(2a - 1)^s} \right]^{s-1}\right)^r.
\end{aligned}
\end{equation}
The upper bound is proven. 

Let us prove the lower bound. 
We define  the number $n \in \NN$, $n \ge 2$, such that  
$N_{n-1} < N + 1 \le N_n$.
From \eqref{ineq[T_{m+1}-T_m]} it follows that $2T_{n-1} > T_n$. Consequently, by 
Corollary \ref{corollary[<d_N<]} and Theorem \ref{thm:inv-bnd-sym} there is a number 
$N^*(s,a) \in \NN$ such that for every  $N \ge N^*(s,a)$,
 \begin{equation}
 \begin{aligned}
&d_N \\
\ &\ge \
T_n^{-r} \\
\ &> \
\big[2 T_{n-1}\big]^{-r} \\
\ &> \
\left(
\left[\frac{N_{n-1}(s-1)!(s+2)}{2^s}\right]^{-1} \left[ \ln \frac{N_{n-1}(s-1)!}{(2a +1)^s} 
- (s-1) \ln \ln \frac{N_{n-1}(s-1)!}{(2a-1)^s} \right]^{s-1}\right)^r \\
\ &> \
2^{-1} \left(
\left[\frac{N_{n-1}(s-1)!(s+2)}{2^s}\right]^{-1} \left[ \ln \frac{N_{n-1}(s-1)!}{(2a +1)^s} 
\right]^{s-1}\right)^r \\
\ &\ge \  
2^{-1} \left( \left[\frac{N(s-1)!(s+2)}{2^s}\right]^{-1} \left[ \ln \frac{N(s-1)!}{(2a +1)^s} 
\right]^{s-1}\right)^r. 
\end{aligned}
\end{equation}
\end{proof}

\begin{theorem} \label{theorem[n_varepsilon(1)]}
Let $r >  0$, $s \ge 2$, $a > 1/2$. Then we have for every $\varepsilon \in (0,[a-1/2]^{-sr})$,
\begin{equation} \label{[n_varepsilon<(1)]}
n_\varepsilon 
\  \leq \ 
\frac{2^s \varepsilon^{-1/r} (\ln \varepsilon^{-1/r} - s\ln (a-1/2))^s }
{(s-1)!\big(\ln \varepsilon^{-1/r} + s\big(1-\ln (a-1/2)) - 1\big)},
\end{equation}
and  for every  $\varepsilon \in (0,[a+1/2]^{-sr})$,
\begin{equation} \label{[n_varepsilon>(1)]}
n_\varepsilon 
\  \geq \ 
\frac{2^s \varepsilon^{-1/r}  (\ln \varepsilon^{-1/r}  - s\ln (a+1/2))^s }
{(s-1)!(\ln \varepsilon^{-1/r}  + s(1-\ln (a + 1/2))} \ - \ 1
\end{equation}
\end{theorem}

\begin{proof}
The upper bound  \eqref{[n_varepsilon<(1)]} can be proven in a way similar to the proof of Theorem 
\ref{theorem[n_varepsilon]} by using Lemmas \ref{lem:IGamma-sym}, \ref{lem:I-val} and Theorem \ref{thm:errTaylor}. 
Let us prove the lower bound \eqref{[n_varepsilon>(1)]}. For every  $\varepsilon \in (0,[a+1/2]^{-sr})$ put
$T = \varepsilon^{-1/r}$.
By Corollary \ref{corollary[<n_e<]} $n_\varepsilon \ge |\Gamma_\pm(s,T)| - 1$. Hence and from Lemmas 
\ref{lem:IGamma-sym}, \ref{lem:I-val} and Theorem \ref{thm:errTaylor} we derive 
\eqref{[n_varepsilon>(1)]}.
\end{proof}

\begin{corollary} \label{corollary[n_varepsilon(1)]}
Let $r >  0$, $s \in \IN$. Then we have for every $a \ge 3/2$ and $\varepsilon \in (0,1]$, 
\begin{equation} 
n_\varepsilon 
\  \leq \ 
\frac{2^s\,r^{-(s-1)}}{(s-1)!} \, \varepsilon^{-1/r}\,|\ln \varepsilon|^{s-1},
\end{equation}
and  for every $1/2 < a < 3/2$ and $\varepsilon \in (0,1]$, 
\begin{equation} 
n_\varepsilon 
\  \leq \ 
\frac{2^s}{(s-1)!} \, \varepsilon^{-1/r}(\ln \varepsilon^{-1/r}  + s|\ln (a - 1/2)|)^{s-1}.
\end{equation}
\end{corollary}

\section{Non-periodic HC approximations}
\label{Non-periodic}

~\newline
The above theory admits an extension to non-periodic functions with slight modifications. In this section we outline the main results.

Suppose $\II^s = [-1,1]^s$ is the reference $s$-dimensional cube and denote by $\Lw$ the Hilbert space of  functions equipped with the weighted inner product
\begin{equation}
 (f,g)_w = \int_{\II^s} f(\bx) g(\bx) w(\bx) d \bx, 
\end{equation}
where the Jacobi weight is give by
\begin{equation}
w(\bx) = \prod_{j=1}^s (1-x_j)^\alpha (1+x_j)^\beta
\end{equation}
for some fixed parameters $\alpha, \beta >-1$. By $\|f\|_w = \sqrt{(f,f)_w}$ we denote the induced norm. Further, let $\{P^{(\alpha,\beta)}_k\}_{k=0}^\infty$ be the family of \emph{orthonormal} Jacobi polynomials, i.e.
\begin{equation}
 \int_{\II} P^{(\alpha,\beta)}_k(x)P^{(\alpha,\beta)}_\ell(x) (1-x)^\alpha (1+x)^\beta dx = \delta_{k\ell}.
\end{equation}
Then
\begin{equation}
 P^{(\alpha,\beta)}_\bk (\bx) = \prod_{j=1}^s  P^{(\alpha,\beta)}_{k_j} (x_j)
\end{equation}
is an orthonormal basis on $\Lw$ and for any $f \in \Lw$ it holds that
\begin{equation}
 \|f\|^2_w = \sum_{k \in \NN_0^s} |f_\bk|^2, \quad \text{where} \quad f_\bk = (f, P^{(\alpha,\beta)}_\bk)_w.
\end{equation}
Define
\begin{equation}
 a := \frac{\alpha + \beta +1}{2}.
\end{equation}
Assuming $a > 0$ we consider the subspaces  
\[
\Kw := \{f \in \Lw : \|f\|_{\Kw} < \infty\}
\]
 of $\Lw$ endowed with the norm
\begin{equation}
 \|f\|_{\Kw}^2 = \sum_{\bk \in \NN_0^s} |\lambda_a(\bk)|^{2r} |f_k|^2.
\end{equation}
To obtain an expression for the above norms in the differential form we recall that Jacobi polynomials $P^{(\alpha,\beta)}_k$ can be characterized as unique solution of the following differential equation, see e.g. \cite[Sect. 4.2]{Sze39}
\begin{equation}
 \cL P^{(\alpha,\beta)}_k(x) = k(k+\alpha+\beta+1) P^{(\alpha,\beta)}_k (x)
\end{equation}
where the differential operator $\cL$ admits the representation
\begin{equation}
 \cL = (1-x)^{-\alpha} (1+x)^{-\beta} \frac{d}{dx} \left((1-x)^{1+\alpha} (1+x)^{1+\beta} \frac{d}{dx} \right).
\end{equation}
This yields for $\cF := a^2 I + \cL$
\begin{equation}
 \cF P^{(\alpha,\beta)}_k(x)
= (k+a)^2P^{(\alpha,\beta)}_k(x)
\end{equation}
implying for $\cF^{(s)} = \cF \otimes \dots \otimes \cF$ ($s$ times)
\begin{equation}
  \cF^{(s)} P^{(\alpha,\beta)}_\bk(\bx)
= \lambda_a(\bk)^2P^{(\alpha,\beta)}_{\bk}(\bx).
\end{equation}
The operator $\cF$ is self-adjoint with respect to $(\cdot,\cdot)_w$
\begin{equation}
 (\cF^{(s)} f, g)_w
= (f,\cF^{(s)}  g)_w
\end{equation}
and it holds that 
\begin{equation} \label{Fsr-def}
 \|f\|_{\Kw}^2 = 
\int_{\II^s} f(\bx) ((\cF^{(s)})^rf)(\bx) w(\bx) d\bx.
\end{equation}

For a function $f \in \Lw $ we define  projection

\begin{equation}
 \Pi_T(f) := \sum_{\bk \in \Gamma(s,T,a)} f_\bk P^{(\alpha,\beta)}_\bk
\end{equation}
and
\begin{equation}
 \cP(s,T,a) := {\rm span}\big\{ P^{(\alpha,\beta)}_\bk ~:~ \bk \in \Gamma(s,T,a)\big\}.
\end{equation}

Similarly to Lemma \ref{lem:Jackson-per} and Lemma \ref{lem:Bernstein-per} we obtain the Jackson and Bernstein inequalities

\begin{lemma} \label{lem:Jackson-nonper}
For arbitrary $T \geq 1$, we have
\begin{equation}
 \|f - \Pi_T(f)\|_{L_2(\IIs,w)} \leq T^{-r} \|f\|_{\Kw} \qquad \forall f \in \Kw
\end{equation}

\end{lemma}

\begin{lemma}\label{lem:Bernstein-nonper}
 For arbitrary $T \geq 1$, we have
\begin{equation}
 \|f\|_{\Kw} \leq T^r  \|f\|_{L_2(\IIs,w)} \qquad \forall f \in  \cP(s,T,a)
\end{equation}
\end{lemma}

Let $\Uw$ be the unit ball in $\Kw$.
Similarly to the periodic case, Jackson and Bernstein inequalities imply. 

\begin{corollary} \label{corollary[<d_N<nonper]}
Let $T \ge 1$ and $N= |\Gamma(s,T,a)|$. Then we have 
\begin{equation} \label{ineq[<d_N<nonper]}
d_N(\Uw, L_2(\II,w)) 
\ \le  \
T^{-r}
\ \le  \
d_{N-1}(\Uw, L_2(\II,w)) .
\end{equation}
\end{corollary}

\begin{corollary} \label{corollary[<n_e<nonper]}
Let $\varepsilon \in (0,1]$. Then we have 
\begin{equation} \label{ineq[<n_e<nonper]}
|\Gamma(s,\varepsilon^{-1/r},a)| - 1
\ \le  \
n_\varepsilon(\Uw, L_2(\II,w))
\ \le  \
|\Gamma(s,\varepsilon^{-1/r},a)|
\end{equation}
\end{corollary}

From these corollaries and the corresponding explicit-in-dimension estimates  of $|\Gamma(s,T,a)|$ in Sections  \ref{sec:def+pre}, \ref{sec:non-asymp} and Subsection \ref{sec:inverted} we prove the following results.

\begin{theorem} \label{theorem[d_N-nonperiodic]}
Let $r > 0$, $s \in \IN$, $a > 0$. Then for any $q \in [1,\infty)$ satisfying  the inequality  
$\lambda:= a - 1/q > 0$, and any $N \in \NN$, we have 
\begin{equation}\label{[d_N-nonperiodic]}
d_N(\Uw, L_2(\II,w)) 
\  \leq \ 
2^rq^{r/(1 + q)}\lambda^{-qrs}N^{-r/(1+q)}.
\end{equation}
\end{theorem}

Similarly to the periodic case, if in Theorem  \ref{theorem[d_N-nonperiodic]} $a>1$ we can choose $q \in [1,\infty)$  so that $\lambda:= a - 1/q > 1$ and then the upper estimate \eqref{[d_N-nonperiodic]} shows that the $N$-widths 
$d_N(\Uw, \Lw)$ is decreasing exponentially in the dimension $s$. This situation corresponds to the exponential tractability of  the problem of $n_\varepsilon(\Uw, \Lw)$ for the case $a>1$ presented in Theorem 
\ref{tractability[nonperiodic]}(i) below.

\begin{theorem}
Let $r >  0$, $s \in \NN$, $a > 0$. 
Then for any $q \in [1,\infty)$ satisfying  the inequality  
$\lambda:= a - 1/q > 0$, and any $\varepsilon \in (0,1]$ we have  
\begin{equation} 
n_\varepsilon(\Uw, \Lw) 
\  \leq \ 
q \lambda^{-qs}\varepsilon^{-(1+q)/r}.
\end{equation}
\end{theorem}

\begin{theorem} \label{tractability[nonperiodic]}
Let $r >  0$, $s \in \IN$, $a > 0$. Then there holds the following.
\begin{itemize}
\item[(i)]
For $a > 1$, the problem of $n_\varepsilon(\Uw, L(\IIs))$ is exponentially tractable. 
 Moreover, for any $q \in [1,\infty)$ satisfying  the inequality  
$\lambda:= a - 1/q > 1$, and any $\varepsilon \in (0,1]$ we have  
\begin{equation}
n_\varepsilon(\Uw, \Lw) 
\  \leq \ 
q \lambda^{-qs}\varepsilon^{-(1+q)/r},
\end{equation}
and 
\begin{equation}
p^{\exp} 
\  \leq \ 
2/r
\end{equation}
\item[(ii)]
For $a = 1$, the problem of $n_\varepsilon(\Uw, \Lw$ is weakly tractable but polynomially intractable. 
\item[(iii)]
For $a < 1$, the problem of $n_\varepsilon(\Uw, \Lw)$ is intractable. 
\end{itemize}
\end{theorem}

\begin{theorem} 
Let $r >  0$, $s \in \IN$, $a > 0$. Then the problem of 
$n_\varepsilon(a^{rs}\Uw, \Lw)$ is weakly tractable but polynomially intractable.
\end{theorem}

\begin{theorem}
Let $r > 0$, $s \geq 2$, $a>1/2$. Then there is a number 
$N^*(a,s)$ such that we have for every $N \ge  N^*(a,s)$,
\begin{equation}
\begin{split}
&2^{-1}\left(\left[N(s-1)!(s+2)\right]^{-1} \left[ \ln \frac{N(s-1)!}{a^s} 
 \right]^{s-1}  \right)^r \\
&\le \
d_N(\Uw, L_2(\IIs,w)) 
\  \leq \ 
2^r \left(\left[N(s-1)!\right]^{-1} \left[ \ln \frac{N(s-1)!}{(a - 1/2)^s} \right]^{s-1}\right)^r.
\end{split}
\end{equation}
\end{theorem}

\begin{theorem} \label{theorem[n_e]}
Let $r >  0$, $s \ge 2$. Then for every $a > 1/2$, there is an $\varepsilon^*= \varepsilon^*(s,a) \in (0,1]$ such that
for every $\varepsilon \in (0,\varepsilon^*]$, 
\begin{equation} 
n_\varepsilon(\Uw, \Lw)) 
\  \leq \ 
\frac{\varepsilon^{-1/r} (\ln \varepsilon^{-1/r} - s\ln (a-1/2))^s }
{(s-1)!\big(\ln \varepsilon^{-1/r} + s\big(1-\ln (a-1/2)) - 1\big)},
\end{equation}
for every $a > 1$ and every $\varepsilon \in (0,(a-1)^{-sr})$,
\begin{equation} 
n_\varepsilon(\Uw, \Lw)) 
\  \leq \ 
\frac{\varepsilon^{-1/r} (\ln \varepsilon^{-1/r} - s\ln (a-1))^s }
{(s-1)!\big(\ln \varepsilon^{-1/r} + s\big(1-\ln (a-1)) - 1\big)},
\end{equation}
and  for every $a > 0$ and every $\varepsilon \in (0,a^{-sr})$,
\begin{equation} 
n_\varepsilon(\Uw, \Lw) 
\  \geq \ 
\frac{\varepsilon^{-1/r}  (\ln \varepsilon^{-1/r}  - s\ln a)^s }
{(s-1)!(\ln \varepsilon^{-1/r}  + s(1-\ln a)} \ - \ 1
\end{equation}
\end{theorem}

\begin{corollary} \label{corollary[n_e-nonperiodic]}
Let $r >  0$, $s \in \IN$. Then we have for every $a \ge 2$ and $\varepsilon \in (0,1]$, 
\begin{equation} 
n_\varepsilon(\tUr, \Lw) 
\  \leq \ 
\frac{r^{-(s-1)}}{(s-1)!} \, \varepsilon^{-1/r}\,|\ln \varepsilon|^{s-1},
\end{equation}
and  for every $1 < a < 2$ and $\varepsilon \in (0,1]$, 
\begin{equation} 
n_\varepsilon(\Uw, \Lw) 
\  \leq \ 
\frac{1}{(s-1)!} \, \varepsilon^{-1/r}(\ln \varepsilon^{-1/r}  + s|\ln (a - 1)|)^{s-1}.
\end{equation}
\end{corollary}

\begin{remark}
When this paper has been written, we received the manuscript \cite{KSU13} related to it.
The authors of \cite{KSU13} investigated approximation numbers of the embedding
the space $H^r_{\operatorname{mix}}(\TTs)$ of mixed smoothness $r$ into $L_2(\TTs)$ emphasizing  the dependence of all constants on the dimension $s$. 
\end{remark}

\section{Appendix: Proof of Theorem \ref{thm:Gamma-asy}} \label{appendix}

\subsection{Auxiliary results}
~\newline
The prove of Theorem \ref{thm:Gamma-asy} relies on a number of auxiliary results summarized in this section. We start with basic notational agreements.
 We denote 
\begin{equation} \label{Gamma0-def}
|\Gamma(0,T,a)| := 1
\end{equation}
for further convenience.
For a function $f: \RR \to \RR$ and $a,b \in \RR$ we abbreviate the sum
\begin{equation}
 \sum_{m=a}^b f(m) := \sum_{m \in M(a,b)} f(m),
\end{equation}
where $M(a,b) := \big\{a+n: n \in \NN_0, \ a+n \leq b\big\}$.
For definiteness, we agree that the sum over an empty set equals zero.

We have the following dimension-by-dimension decomposition
\begin{equation}\label{Gamma-d1}
\begin{split}
|\Gamma(s,T,a)|
& \ = \
\sum_{j=0}^s
\#\big\{ \bk \in \Gamma(s,T,a): |\supp \bk| = j \big\}\\
& \ = \
\sum_{j=0}^s  \binom{s}{j} \# \bigg\{ \bk \in \NN^j : \prod_{i=1}^j (k_i+a)\leq T a^{j-s} \bigg\} \\
& \ = \
\sum_{j=0}^s  \binom{s}{j} |\Gamma(j,Ta^{j-s},a+1)|.
\end{split}
\end{equation}
Note that 
\begin{equation}
 |\Gamma(j,Ta^{j-s},a+1)| \neq 0 \quad \Leftrightarrow \quad (a+1)^j \leq T a^{j-s}.
\end{equation}
This implies in particular that all summands in \eqref{Gamma-d1} are nonzero if and only if the term with $j=s$ is nonzero. Extracting this last term we apply the decomposition \eqref{Gamma-d1} again and obtain for $ M = a,a+1\dots$
\begin{equation}\label{Gamma-d2}
\begin{split}
|\Gamma(s,T,a)| 
&\ = \ 
|\Gamma(s,T,M +1)| + 
\sum_{m=a}^{M} \sum_{j=0}^{s-1} \binom{s}{j}   |\Gamma(j,Tm^{j-s},m+1)|.
\end{split}
\end{equation}
The first term on the right-hand side vanishes if $M$ satisfies $(M+1)^s > T$. 
Suppose $a^s\leq T$ and let $n_0$ be the largest nonnegative integer such that $n_0 \leq T^{1/s}-a$. Then for $M_0 := a+ n_0$ it holds that $M_0 \leq T^{1/s}$ and $M_0+1 > T^{1/s}$ implying $|\Gamma(s,T,M +1)| = 0$ and therefore

\begin{equation} \label{Gamma-deco}
\begin{split}
|\Gamma(s,T,a)| 
&\ = \
\sum_{m=a}^{T^{1/s}} \sum_{j=0}^{s-1} \binom{s}{j}   |\Gamma(j,Tm^{j-s},m+1)|.
\end{split}
\end{equation}

Notice that the right-hand side of \eqref{Gamma-deco} might contain zero summands. Moreover, decomposition \eqref{Gamma-deco} is formally true also when $|\Gamma(s,T,a)|=0$. Indeed, in this case $a^s > T$ and the summation over $m$ in \eqref{Gamma-deco} runs over an empty set, which we formally interpret as zero.

We have a similar decomposition for the integral. Denote
$\II:=[0,1]$ and set for the notational convenience 
\begin{equation}\label{I0-def}
I(0,T,a)):=1. 
\end{equation}
Then for any $a> \tfrac{1}{2}$
\begin{equation} \nonumber
\begin{split}
I(s,T,a-\tfrac{1}{2})
\ = \
I(s,T,a+\tfrac{1}{2}) 
&+ \sum_{j=0}^{s-1} \binom{s}{j} \int_{\bt \in \II^{s-j}} I(j,T\prod_{\ell=1}^{s-j}(t_\ell+a-\tfrac{1}{2})^{-1}, a+\tfrac{1}{2})\, \bd \bt
\end{split}
\end{equation}
and consequently, for $\tilde m = a,a+1\dots$
\begin{equation} \label{eq[I(s,T,a+3/2)]}
I(s,T,a-\tfrac{1}{2})
\ = \
I(s,T,\tilde m+\tfrac{1}{2})+
\sum_{m=a}^{\tilde m}
\sum_{j=0}^{s-1}  \binom{s}{j} \int_{\bt \in\II^{s-j}}  
I(j,T\prod_{\ell=1}^{s-j}(t_\ell+m-\tfrac{1}{2})^{-1}, m+\tfrac{1}{2})\, \bd \bt.
\end{equation}
The first term on the right-hand side first is zero if $(\tilde m+\tfrac{1}{2})^s \geq T$. Thus, setting $\tilde m := T^{1/s}$
\begin{equation}\label{I-deco}
 I(s,T,a-\tfrac{1}{2})
\ = \
\sum_{m = a}^{T^{1/s}} \sum_{j=0}^{s-1} \binom{s}{j}   
\int_{\bt \in\II^{s-j}}  
I(j,T\prod_{\ell=1}^{s-j}(t_\ell+m-\tfrac{1}{2})^{-1}, m+\tfrac{1}{2})\, \bd \bt.
\end{equation}
Note that the above sum may contain some zero summands.

The proof of Theorem \ref{thm:Gamma-asy} will be based on the comparison of $|\Gamma(s,T,a)|$ and $I(s,T,a-\tfrac{1}{2})$ involving their similar series representations \eqref{Gamma-deco} and \eqref{I-deco}. The fundamental idea of the proof relies on the structure of \eqref{Gamma-deco} and \eqref{I-deco} allowing to represent the volume of the continuous and smooth $s$-dimensional HCs as a sum of volumes of HCs of lower dimensions: $j$-dimensional crosses for $j=0, \dots, s-1$. To simplify this quite technical comparison, we introduce several auxiliary sequences depending on the parameter $T$ and the  vectors 
$[\bj]_n := (j_0,\dots,j_n) \in \IN^{n+1}_0$ and $[\bm]_n := (m_0,\dots,m_n) \in \IR^{n+1}_+$ for 
$\bj = (j_0,...,j_{s-1}) \in \NN^s_0$, $\bm = (m_0,...,m_{s-1}) \in \RR^s_+$ and $0 \leq n\leq s-1$. 
In what follows the first components in $[\bj]_n$ and $[\bm]_n$ will be always associated with the parameters $s$ and $a$ respectively:
\begin{equation}\label{j0m0-def}
 j_0 := s, \qquad m_0 := a. 
\end{equation}
This redundant notation will greatly simplify the forthcoming expressions.

For fixed $[\bj]_n, [\bm]_n$ and $T$ we define
\begin{equation}\label{Tdef}
T_\ind\big([\bj]_n,T,[\bm]_n\big) := \left\{ 
\begin{array}{cc}
 T, & n=0,\\[1.5ex]
T_{\ind-1}\big([\bj]_n,T,[\bm]_n\big) (m_\ind)^{j_\ind - j_{\ind-1}}, & 1 \leq n \leq s-2.
\end{array}
\right.
\end{equation}
In what follows we will skip the arguments and write $T_\ind \equiv T_\ind([\bj]_n,T,[\bm]_n)$ to simplify the notations.
For fixed $[\bj]_n,m_0$, $T$ and any $1 \leq n \leq s-1$ we introduce the partial sums
\begin{equation}\label{Xdef}
X_\ind([\bj]_n,T,m_0) := 
\left[\binom{j_0}{j_1} \dots \binom{j_{\ind-1}}{j_\ind}\right]
\sum_{m_1=m_0+1}^{T_0^{1/j_0}} \dots 
\sum_{m_\ind=m_{\ind-1}+1}^{T_{\ind-1}^{1/j_{\ind-1}}}
|\Gamma(j_\ind,T_\ind,m_\ind+1)|.
\end{equation}
We observe that \eqref{Gamma-deco} and \eqref{Xdef} imply the relation
\begin{equation}\label{X-rel}
 X_\ind([\bj]_n,T,m_0) = 
\sum_{j_{\ind+1}=0}^{j_\ind-1} X_{\ind+1}([\bj]_{n+1},T,m_0).
\end{equation}
Finally, we introduce the sums
\begin{equation} \label{ABC-def}
\begin{split}
C_\ind(j_0,T,m_0)
&:= 
\sum_{j_1 = \ind+1}^{j_0-1} 
\sum_{j = \ind}^{j_1-1}\dots \sum_{j_{\ind-1}=3}^{j_{\ind-2}-1}
\sum_{j_{\ind}=2}^{j_{\ind-1}-1} 
X_\ind([\bj]_n,T,m_0), \\
B_\ind(j_0,T,m_0)
&:= 
~\sum_{j_1 = \ind}^{j_0-1} 
\sum_{j = \ind-1}^{j_1-1}\dots \sum_{j_{\ind-1}=2}^{j_{\ind-2}-1}
\sum_{j_{\ind}=1}^{1} 
X_\ind([\bj]_n,T,m_0), \\
A_\ind(j_0,T,m_0)
&:= 
~\sum_{j_1 = \ind}^{j_0-1} 
\sum_{j = \ind-1}^{j_1-1}\dots \sum_{j_{\ind-1}=2}^{j_{\ind-2}-1}
\sum_{j_{\ind}=0}^{0} 
X_\ind([\bj]_n,T,m_0).
\end{split}
\end{equation}
Note that the sums over $j_n$ in $A_n, B_n$ are trivial and consist only of one term with $j_n = 0$ and $1$ respectively. With this notations, we obtain the following representation for the cardinality  of the continuous HC $\Gamma(s,T,a+1)$.

\begin{lemma} \label{lem:AB-deco}
Assume that $T>0$, $a+1>0$ and $s \geq 2$. Then
\begin{equation}\label{AB-deco}
 |\Gamma(s,T,a+1)| = \sum_{\ind=1}^{s-1} \big[A_\ind(s,T,a) + B_\ind(s,T,a)\big].
\end{equation}
\end{lemma}

\begin{proof}
The relations \eqref{Gamma-deco} and \eqref{ABC-def} imply
\begin{equation}\label{step}
|\Gamma(j_0,T_0,m_0+1)| 
 = 
\sum_{m_1=m_0+1}^{T_0^{1/j_0}} \sum_{j_1=0}^{j_0-1} \binom{j_0}{j_1}
|\Gamma(j_1,T_1,m_1+1)| = A_1 + B_1 + C_1.
\end{equation}
We claim
\begin{equation}\label{C-ABC-rel}
 C_{\ind} = 
\left\{
\begin{array}{lr}
A_{\ind+1} + B_{\ind+1} + C_{\ind+1}, & 1 \leq \ind \leq s-3, \\[1.5ex]
A_{\ind+1} + B_{\ind+1}, & n = s-2.
\end{array}
\right.
\end{equation}
Indeed, by \eqref{X-rel} and \eqref{ABC-def} we have for $1 \leq \ind \leq s-3$
\begin{equation}
\begin{split}
C_\ind(j_0,T,m_0)
&= 
\sum_{j_1 = \ind+1}^{j_0-1} 
\sum_{j = \ind}^{j_1-1}\dots \sum_{j_{\ind-1}=3}^{j_{\ind-2}-1}
\sum_{j_{\ind}=2}^{j_{\ind-1}-1} 
X_\ind([\bj]_n,T,m_0)\\
&= 
\sum_{j_1 = \ind+1}^{j_0-1} 
\sum_{j = \ind}^{j_1-1}\dots \sum_{j_{\ind-1}=3}^{j_{\ind-2}-1}
\sum_{j_{\ind}=2}^{j_{\ind-1}-1} 
\sum_{j_{\ind+1}=0}^{j_{\ind}-1} 
X_{\ind+1}([\bj]_{n+1},T,m_0)\\
&= A_{\ind+1} + B_{\ind+1} + C_{\ind+1}
\end{split}
\end{equation}
and
\begin{equation}
\begin{split}
C_{s-2}
&= 
\sum_{j_1 = s-1}^{s-1} 
\sum_{j = s-2}^{j_1-1}\dots \sum_{j_{s-3}=3}^{j_{s-4}-1}
\sum_{j_{s-2}=2}^{j_{s-3}-1} 
X_{s-2}([\bj]_{s-2},T,m_0)\\
&= 
\sum_{j_1 = s-1}^{s-1} 
\sum_{j = s-2}^{s-2}\dots \sum_{j_{s-3}=3}^{3}
\sum_{j_{s-2}=2}^{2} 
\sum_{j_{s-1}=0}^{1} 
X_{s-1}([\bj]_{s-1},T,m_0)\\
&= A_{s-1} + B_{s-1}.
\end{split}
\end{equation}
Therefore, using \eqref{step} and \eqref{C-ABC-rel} repeatedly for $n=1,\dots,s-2$ we obtain the assertion.
\end{proof}

Next, we obtain lower bounds for the volume of the smooth HC $P(s,T,a)$. For this, we need more auxiliary sequences. For fixed $[\bj]_\ind, [\bm]_\ind$, $T$ and $\bt \in \IR^{j_{\ind-1} - j_\ind}$ we define 
\begin{equation}\label{R-def}
R_\ind([\bj]_\ind,T,[\bm]_\ind,\bt) 
:= \left\{
\begin{array}{cc}
T, & n=0, \\[1.5ex]
T_\ind([\bj]_\ind,T,[\bm]_\ind) \displaystyle
\prod_{\ell=1}^{j_{\ind-1} - j_\ind} (t_\ell+m_\ind-\tfrac{1}{2})^{-1}, & 1 \leq n \leq s-2.
\end{array}
\right.
\end{equation}
To shorten the notations we write $R_\ind(\bt) \equiv R_\ind([\bj]_\ind,T,[\bm]_\ind,\bt) $
and for $1 \leq n \leq s-1$ the partial sums
\begin{equation}\label{Ydef}
\begin{split}
Y_\ind([\bj]_n,T,m_0):= 
\left[\binom{j_0}{j_1} \dots \binom{j_{\ind-1}}{j_\ind}\right]
\sum_{m_1=m_0+1}^{T_0^{1/j_0}} \dots 
\sum_{m_\ind=m_{\ind-1}+1}^{T_{\ind-1}^{1/j_{\ind-1}}} \times \\
\times \int \limits_{\bt \in \II^{j_{\ind-1}-j_\ind}} I(j_\ind,R_\ind(\bt),m_\ind + \tfrac{1}{2}),
\end{split}
\end{equation}
We claim that for $2 \leq j_n < j_{n-1}< \dots < j_0$ and $m_0+\frac{1}{2}>0$ it holds that
\begin{equation} \label{Y-rel}
Y_\ind([\bj]_n,T,m_0) > \sum_{j_{\ind+1}=0}^{j_\ind-1} Y_{\ind+1}([\bj]_{n+1},T,m_0).
\end{equation}
This estimate follows from the strict convexity of the volume of  the $j$th dimensional smooth HC $P(j,t,a)$
($2 \leq j \leq s-1$) and relies on the following lemma.

\begin{lemma}\label{lemma[intI>]}
For any natural numbers $j,s$ satisfying $2 \leq j \leq s-1$ and $T>0$, $a - \tfrac{1}{2} > 0$ it holds that
\begin{equation}\label{ineq[intI>]}
\int_{\bt \in \II^{s-j}} I(j,T \prod_{\ell = 1}^{s-j}(t_\ell+a-\tfrac{1}{2})^{-1},a+\tfrac{1}{2})\, \bd \bt
> I(j,T a^{j-s},a+\tfrac{1}{2}).
\end{equation}
\end{lemma}

\begin{proof}
By Lemma \ref{lem:I-val}, we can verify that for fixed $j$ and $a$, the first and the second derivatives of the function $\Phi(T):= I(j,T,a+\tfrac{1}{2})$ are positive for $T > (a+\tfrac{1}{2})^j$. Hence, $\Phi$ is increasing and strictly convex on the interval $\big((a+\tfrac{1}{2})^j, \infty\big)$.
Consider 
\begin{equation}
 u(\bt) := T\prod_{\ell=1}^{s-j}(t_\ell+a-\tfrac{1}{2})^{-1}, \qquad G(\bt) := \Phi(u(\bt)).
\end{equation}
Since the function $u(\bt)$ is strictly convex on $\II^{s-j}$ in all variables $\bt = (t_1,\dots,t_{s-j})$, so is $G(\bt)$. 
For a strictly convex function $g$ on $[0,1]$ we have 
\begin{equation} 
g(1/2)
\ < \
\frac{1}{2}\int_0^1  [g(x) + g(1-x)] \, dx 
= \int_0^1 g(x) \, dx.
\end{equation}
Applying this inequality to the $(s-j)$-variate function $G$ as a univariate function in each variable $t_i$ consecutively with the other variables held fixed, we get
\begin{equation*} 
\begin{split}
\int_{\II^{s-j}}  I(j,T\prod_{\ell=1}^{s-j}(t_\ell+a-\tfrac{1}{2})^{-1}, a+\tfrac{1}{2})\, \bd \bt 
\ &= \ \int_{\II^{s-j}} G(\bt) \, \bd \bt \\[2ex]
\ &> \
\int_{\II^{s-j-1}} G(\tfrac{1}{2},t,...,t_{s-j}) \, dt...dt_{s-j} \\[2ex]
&\cdots \cdots \cdots \\
\ &> \
G(\tfrac{1}{2},...,\tfrac{1}{2})
\ = \ 
I(j,Ta^{j-s}, a+\tfrac{1}{2}).
\end{split}
\end{equation*}
The proof is complete.
\end{proof}

Now, it is easy to verify that the combination of \eqref{I-deco} and \eqref{ineq[intI>]} implies the relation \eqref{Y-rel}. We leave the details to the interested reader. 

Similarly to \eqref{ABC-def} we introduce the sums
\begin{equation} \label{DEF-def}
\begin{split}
F_\ind(j_0,T,m_0)
&:= 
\sum_{j_1 = \ind+1}^{j_0-1} 
\sum_{j = \ind}^{j_1-1}\dots \sum_{j_{\ind-1}=3}^{j_{\ind-2}-1}
\sum_{j_{\ind}=2}^{j_{\ind-1}-1} 
Y_\ind([\bj]_n,T,m_0), \\
E_\ind(j_0,T,m_0)
&:= 
~\sum_{j_1 = \ind}^{j_0-1} 
\sum_{j = \ind-1}^{j_1-1}\dots \sum_{j_{\ind-1}=2}^{j_{\ind-2}-1}
\sum_{j_{\ind}=1}^{1} 
Y_\ind([\bj]_n,T,m_0), \\
D_\ind(j_0,T,m_0)
&:= 
~\sum_{j_1 = \ind}^{j_0-1} 
\sum_{j = \ind-1}^{j_1-1}\dots \sum_{j_{\ind-1}=2}^{j_{\ind-2}-1}
\sum_{j_{\ind}=0}^{0} 
Y_\ind([\bj]_n,T,m_0).
\end{split}
\end{equation}

\begin{lemma}\label{lem:DE-deco}
Assume that $T> (a+\tfrac{1}{2})^s$, $a+\tfrac{1}{2}> 0$ and $s \geq 2$. Then
\begin{equation}\label{DE-deco}
 I(s,T,a+\tfrac{1}{2}) \ > \ \sum_{\ind = 1}^{s-1} \big[D_\ind(s,T,a) + E_\ind(s,T,a) \big].
\end{equation}
\end{lemma}

\begin{proof}
Observe that relation \eqref{I-deco} and definitions \eqref{R-def}, \eqref{Ydef}, \eqref{DEF-def} imply
\begin{equation}\label{step-I}
\begin{split}
& I(j_0,R_0,m_0+\tfrac{1}{2}) \\
&\ = \
\sum_{m_1 = m_0+1}^{T_0^{1/j_0}} \sum_{j_1=0}^{j_0-1} \binom{j_0}{j_1}   
\int \limits_{\bt \in\II^{j_0-j_1}}  
I(j_1,T_0\prod_{\ell=1}^{j_0-j_1}(t_\ell+m_1 - \tfrac{1}{2})^{-1}, m_1+\tfrac{1}{2})\, \bd \bt \\
&\ = D_1 + E_1 + F_1.
\end{split}
\end{equation}
Furthermore, we claim that
\begin{equation}\label{F-DEF-rel}
 F_\ind > 
\left\{
\begin{array}{lr}
D_{\ind+1} + E_{\ind+1} + F_{\ind+1}, & 1 \leq \ind 
\leq s-3, \\[1.5ex]
D_{\ind+1} + E_{\ind+1}, & \ind 
= s-2.
\end{array}
\right.
\end{equation}
Indeed, by \eqref{DEF-def} and \eqref{Y-rel}
\begin{equation}
\begin{split}
 F_\ind(j_0,T,m_0)
&= 
\sum_{j_1 = \ind+1}^{j_0-1} 
\sum_{j = \ind}^{j_1-1}\dots \sum_{j_{\ind-1}=3}^{j_{\ind-2}-1}
\sum_{j_{\ind}=2}^{j_{\ind-1}-1} 
Y_\ind([\bj]_n,T,m_0) \\
&> 
\sum_{j_1 = \ind+1}^{j_0-1} 
\sum_{j = \ind}^{j_1-1}\dots \sum_{j_{\ind-1}=3}^{j_{\ind-2}-1}
\sum_{j_{\ind}=2}^{j_{\ind-1}-1} 
\sum_{j_{\ind+1}=0}^{j_{\ind}-1} 
Y_{\ind+1}([\bj]_{\ind+1},T,m_0) 
\end{split}
\end{equation}
and \eqref{F-DEF-rel} follows. Using \eqref{step-I} and \eqref{F-DEF-rel} repeatedly for $\ind = 1,\dots,s-2$ implies the assertion.
\end{proof}

To prove the main result in Theorem \ref{thm:Gamma-asy} we will compare $|\Gamma(s,T,a+1)|$ with $I(s,T,a+\tfrac{1}{2})$ by using the representation in \eqref{AB-deco} and the inequality in \eqref{DE-deco}. More precisely, we will compare $B_\ind(s,T,a)$ with the corresponding term $E_\ind(s,T,a)$ and $A_\ind(s,T,a)$ with the corresponding term $D_\ind(s,T,a)$ due to the fact that the summation ranges in the right side of in \eqref{AB-deco} and \eqref{DE-deco} are the same. It is easy to see that \eqref{Gamma0-def} and \eqref{I0-def} imply 
\begin{equation}\label{A=D}
 A_\ind(s,T,a) = D_\ind(s,T,a).
\end{equation}
The relation between $B_\ind(s,T,a)$ and $E_\ind(s,T,a)$ is less obvious and requires the following auxiliary lemma.

\begin{lemma} \label{lem:1d-vol}
For $1\le \ind \le s-1$ and $j_\ind = 1$ it holds that
\begin{equation} \label{Gamma1d}
|\Gamma\big(1,T_\ind ,m_\ind+1\big)|
\ \le \ 
T_{\ind-1}(m_\ind)^{1 - j_{\ind-1}} - (m_\ind + 1),
\end{equation}
\begin{equation} \label{I1d}
\int_{\bt \in \II^{s-1}} I\big(1,R_\ind,m_n+\tfrac{1}{2}\big) \, \b\bd \bt\\
\ = \
T_{\ind-1}
\left(\ln \frac{m_\ind-\tfrac{1}{2}}{m_\ind+\tfrac{1}{2}} \right)^{1-j_{\ind-1}}- (m_\ind +\tfrac{1}{2}),
\end{equation}
where $T_\ind = T_\ind([\bj]_n,T,[\bm]_n)$ and $R_\ind(\bt) = R_\ind([\bj]_n,T,[\bm]_n, \bt)$ are defined in \eqref{Tdef} and \eqref{R-def} respectively.
\end{lemma}

\begin{proof}
For $1\leq \ind \leq s-1$ we have 
\begin{equation} \label{eq[integral]}
|\Gamma(1,T_\ind,m_\ind+1)| 
\ = \
\lfloor T_\ind\rfloor - (m_\ind + 1)
\leq T_\ind - (m_\ind + 1).
\end{equation}
The bound \eqref{Gamma1d} from the definition \eqref{Tdef} yielding
\begin{equation}
T_\ind
= T_{\ind-1} (m_\ind)^{1 - j_{\ind-1}}.
\end{equation}
Recalling \eqref{R-def} and $j_\ind = 1$ we obtain for the integral 
\begin{equation}
\begin{split}
 \int_{\bt \in \II^{s-1}} I(1,R_\ind,m_\ind+\tfrac{1}{2}) \, \b\bd \bt 
\ &= \
\int_{\bt \in \II^{s-1}} 
T_{\ind-1} \prod_{\ell = 1}^{j_{\ind-1}-1} (t_\ell + m_\ind-\tfrac{1}{2})^{-1}\, \b\bd \bt 
- (m_\ind + \tfrac{1}{2}) \\[2ex]
\ &= \
T_{\ind-1}
\left(\ln \frac{m_\ind+\tfrac{1}{2}}{m_\ind-\tfrac{1}{2}} \right)^{j_{\ind-1}- 1}- (m_\ind +\tfrac{1}{2}),
\end{split}
\end{equation}
and hence the assertion \eqref{I1d}.
\end{proof}

\subsection{Proof of Theorem \ref{thm:Gamma-asy}}

~\newline
We are now in a position to prove Theorem \ref{thm:Gamma-asy}.

\proofof{Theorem \ref{thm:Gamma-asy}} It is convenient to us to prove the theorem in the following equivalent form.
For any $s \in \NN$, and $a+\tfrac{1}{2}>0$ there exists $T_* = T_*(s,a) >0$ such that
\begin{equation} \label{Gamma-asy(1)}
 |\Gamma(s,T,a+1)| \leq I(s,T,a+\tfrac{1}{2}), \quad
\forall T \geq T_*(s,a).
\end{equation}

Define for $1\leq \ind \leq s-1$ and positive integers 
\begin{equation} \label{jm-fam}
s = j_0 > j_1> \dots > j_{\ind-1} > j_\ind=1, \qquad 
a = m_0 < m_1< \dots < m_{\ind-1} < m_\ind
\end{equation}
an auxiliary family
\begin{equation} 
L_\ind([\bj]_n,[\bm]_n)
:= \
\left(\prod_{\ell=1}^{\ind-1} m_\ell^{j_\ell-j_{\ell-1}} \right) \left(\left(\ln \frac{m_\ind-\tfrac{1}{2}}{m_\ind+\tfrac{1}{2}} \right)^{1-j_{\ind-1}} - m_\ind^{1 - j_{\ind-1}} \right).
\end{equation} 
Using Lemma \ref{lem:AB-deco} and Lemma \ref{lem:DE-deco} and observing \eqref{A=D} we get
\begin{equation}\label{I-Gamma=sum}
I(s,T,a+\tfrac{1}{2}) - |\Gamma(s,T,a+1)|
> \sum_{\ind=1}^{s-1} (E_\ind - B_\ind).
\end{equation}

The summands on the right-hand side of \eqref{I-Gamma=sum} are estimated by Lemma \ref{lem:1d-vol} as follows
\begin{equation}\label{sumEBlow}
\begin{split}
 E_\ind - B_\ind
=& \sum_{j_1 = \ind}^{j_0-1} 
\sum_{j = \ind-1}^{j_1-1}\dots \sum_{j_{\ind-1}=2}^{j_{\ind-2}-1}
\sum_{j_{\ind}=1}^{1} 
\bigg(
Y_\ind([\bj]_n,T,m_0) 
- X_\ind([\bj]_n,T,m_0) 
\bigg) \\
=& \sum_{j_1 = \ind}^{j_0-1} 
\sum_{j = \ind-1}^{j_1-1}\dots \sum_{j_{\ind-1}=2}^{j_{\ind-2}-1}
\sum_{j_{\ind}=1}^{1} 
\left[
\binom{j_0}{j_1} \dots
\binom{j_{\ind-1}}{j_\ind}
\right] \times \\
& \times 
\sum_{m_1=m_0+1}^{T_0^{1/j_0}} \dots 
\sum_{m_\ind = m_{\ind-1}+1}^{T_{\ind-1}^{1/j_{\ind-1}}}
\bigg( T L_\ind([\bj]_n,[\bm]_n) - \tfrac{1}{2} \bigg)
 \\
=& \ T \alpha_\ind(s,a) - \frac{1}{2} \beta_\ind(s,T,a),
\end{split}
\end{equation}
where $\alpha_\ind(s,a)$ and $\beta_\ind(s,T,a)$ are defined as
\begin{equation}\label{alpha-beta-def}
\begin{split}
\alpha_\ind(s,a) := 
\sum_{j_1=\ind}^{j_0-1} \dots \sum_{j_{\ind-1}=2}^{j_{\ind-2}-1}
\left[ \binom{j_0}{j_1} \dots \binom{j_{\ind-2}}{j_{\ind-1}}\binom{j_{\ind-1}}{1}\right] \sum_{m_1=m_0+1}^{T_0^{1/j_0}} \dots 
\sum_{m_\ind = m_{\ind-1}+1}^{T_{\ind-1}^{1/j_{\ind-1}}} \times \\
\times L_\ind([\bj]_\ind,[\bm]_\ind), \\
\beta_\ind(s,T,a)
:=
\sum_{j_1=\ind}^{j_0-1} \dots \sum_{j_{\ind-1}=2}^{j_{\ind-2}-1}
\left[ \binom{j_0}{j_1} \dots \binom{j_{\ind-2}}{j_{\ind-1}}\binom{j_{\ind-1}}{1}\right]
\sum_{m_1=m_0+1}^{T_0^{1/j_0}} \dots 
\sum_{m_\ind = m_{\ind-1}+1}^{T_{\ind-1}^{1/j_{\ind-1}}} 1.
\end{split}
\end{equation}
The right-hand side of \eqref{sumEBlow} is positive for a sufficiently large $T \geq T_*(s,a)$ if $\alpha(s,a)>0$ and $\beta_\ind(s,T,a) = o(T)$ for a fixed $s$ and $T \to \infty$, which we will prove next.

Positivity of $\alpha(s,a)$ follows from positivity of $L_\ind([\bj]_\ind,[\bm]_\ind)$ for the asserted range of indices \eqref{jm-fam}. Indeed, it follows by convexity of $g(x) = \tfrac{1}{x+m-1/2}$: For any $m > \tfrac{1}{2}$
\begin{equation} 
\frac{1}{m}
\ = \
g(1/2)
\ < \
\frac{1}{2}\int_0^1  [g(x) + g(1-x)] \, dx 
\ = \ \int_0^1 g(x) \, dx 
\ = \ \ln \frac{m+\tfrac{1}{2}}{m-\tfrac{1}{2}}.
\end{equation}
This implies
\begin{equation}
\left(\ln \frac{m_\ind+\tfrac{1}{2}}{m_\ind-\tfrac{1}{2}} \right)^{j_{\ind-1}-1} - \left(\frac{1}{m_\ind}\right)^{j_{\ind-1}-1} > 0,
\end{equation}
since $j_{\ind-1}-1 \geq 1$. Therefore $\alpha_\ind(s,a)$ is positive for any $1 \leq \ind \leq s-1$. 

For a given 
$1\le \ind \le s-1$, in order to prove upper bounds for $\beta_\ind(s,T,a)$, we fix a multiindex $[\bj]_{n-1} = (j_0,...,j_{\ind-1})$ in the outer summation of $\beta_\ind(s,T,a)$ satisfying   \eqref{jm-fam},
and consider the inner summation admitting the simple upper bound
\begin{equation}
\begin{split}
\sum_{m_1=m_0+1}^{T_0^{1/j_0}} 
\dots 
\sum_{m_\ind = m_{\ind-1}+1}^{T_{\ind-1}([\bj]_{\ind-1},[\bm]_{\ind-1})^{1/j_{\ind-1}}}1
\leq 
\sum_{m_1=1}^{T_0^{1/j_0}} 
\dots 
\sum_{m_\ind =1}^{T_{\ind-1}([\bj]_{\ind-1},[\bm]_{\ind-1})^{1/j_{\ind-1}}}1\\
=: P_\ind([\bj]_{\ind-1},T).
\end{split}
\end{equation}
We claim that for all $1 \leq \ind \leq s-1$, 
\begin{equation}\label{Pclaim}
 P_\ind([\bj]_{\ind-1},T) \leq C_\ind(s,[\bj]_{\ind-1}) T^{(s-1)/s},
\end{equation}
where
\begin{equation}\nonumber
C_1(s,[\bj]_0):=\ 1, \quad  C_\ind(s,[\bj]_{\ind-1}):= \ \prod_{i=1}^{\ind-1}\lambda_i([\bj]_i),
\end{equation}
\begin{equation}\nonumber
\lambda_i([\bj]_i)
:= \
\begin{cases}
\left[1 + \frac{(j_i-1)(j_i-j_{i-1})}{j_i+(j_i-1)(j_i-j_{i-1})} \right], \ & (j_i-1)(j_i-j_{i-1})/j_i \not= -1, \\
j_{i-1} - 1, \ & (j_i-1)(j_i-j_{i-1})/j_i = -1.
\end{cases}
\end{equation}
The proof is by induction over $\ind$. Obviously, 
\begin{equation}
P_1([\bj]_0,T) = \lfloor T_0^{1/j_0} \rfloor \leq T^{1/s} \le T^{(s-1)/s}.
\end{equation}
Assume now that \eqref{Pclaim} holds true and observe
\begin{equation} \label{ineq[estimationP_nu(0)]}
\begin{split}
 P_{\ind+1}([\bj]_n,T)
&= \sum_{m_1=1}^{T_0^{1/j_0}} \sum_{m=1}^{T_1^{1/j_1}} 
\dots 
\sum_{m_{\ind+1} =1}^{T_{\ind}([\bj]_n,[\bm]_n)^{1/j_{\ind}}}1\\
&\leq C_\ind'(s,[\bj]_{\ind-1}) \sum_{m_1=1}^{T_0^{1/j_0}}  T_1([\bj]_1,[\bm]_1)^{(j_1-1)/j_1}
\end{split}
\end{equation}
by inductive assumption, where
\begin{equation}\nonumber
C_1'(s,[\bj]_0):=\ 1, \quad  C_\ind'(s,[\bj]_{\ind-1}):= \ \prod_{i=2}^{\ind-1}\lambda_i([\bj]_i),
\end{equation}
It holds that $T_1([\bj]_1,[\bm]_1) = T_0 m_1^{j_1-j_0}$ and the function $x ^{(j_1-1)(j_1-j_0)/j_1}$ is decreasing. Thus
\begin{equation} \label{ineq[estimationP_nu(1)]}
\begin{split}
\sum_{m_1=1}^{T_0^{1/j_0}} T_1([\bj]_1,[\bm]_1)^{(j_1-1)/j_1}
&= \sum_{m_1=1}^{T_0^{1/j_0}} T_0^{(j_1-1)/j_1} m_1^{(j_1-1)(j_1-j_0)/j_1}\\[2ex]
&\leq T_0^{(j_1-1)/j_1} 
\left(1 + \int_{1}^{T_0^{1/j_0}} x ^{(j_1-1)(j_1-j_0)/j_1} dx \right)
\end{split}
\end{equation}
If $(j_1-1)(j_1-j_0)/j_1 \not= -1$, we have
\begin{equation} \label{ineq[estimationP_nu(2)]}
\begin{split}
\sum_{m_1=1}^{T_0^{1/j_0}} T_1([\bj]_1,[\bm]_1)^{(j_1-1)/j_1}
\ &\le \
 T_0^{(j_1-1)/j_1} \left(1 + \frac{x^{1+\frac{(j_1-1)(j_1-j_0)}{j_1}}}{1+\frac{(j_1-1)(j_1-j_0)}{j_1}}\bigg|_{1}^{T_0^{1/j_0}} \right)\\[2ex]
\ &= \  T_0^{(j_1-1)/j_1} 
\left(T_0^{\frac{j_1+(j_1-1)(j_1-j_0)}{j_0j_1}}
+ \frac{(j_1-1)(j_1-j_0)}{j_1+(j_1-1)(j_1-j_0)}  \right)
 \\[2ex]
\ &= \
T_0^{j_1/j_0} + T_0^{(j_1-1)/j_1} \frac{(j_1-1)(j_1-j_0)}{j_1+(j_1-1)(j_1-j_0)} \\[2ex]
&\leq  \left[1 + \frac{(j_1-1)(j_1-j_0)}{j_1+(j_1-1)(j_1-j_0)}\right] T^{(s-1)/s}.
\end{split}
\end{equation}
If $(j_1-1)(j_1-j_0)/j_1 = -1$, we have
\begin{equation} \label{ineq[estimationP_nu(3)]}
\begin{split}
\sum_{m_1=1}^{T_0^{1/j_0}} T_1([\bj]_1,[\bm]_1)^{(j_1-1)/j_1}
&= T_0^{(j_1-1)/j_1} \left(1 + \ln x \bigg|_{1}^{T_0^{1/j_0}} \right)\\
&=  T_0^{(j_1-1)/j_1} \left( 1 + \frac{1}{j_0} \ln T_0 \right)  \\
&\leq  
T^{(s-2)/(s-1)} \left( 1 + \frac{1}{j_0} \ln T \right)  \\
&\leq  (j_0-1) T^{(s-1)/s}.
\end{split}
\end{equation}
In \eqref{ineq[estimationP_nu(2)]} and \eqref{ineq[estimationP_nu(3)]} we used the relation $j_1<j_0=s$ and the fact that the function $(x-1)/x$ is monotonously increasing for $x>1$.  
Combining \eqref{ineq[estimationP_nu(0)]}--\eqref{ineq[estimationP_nu(3)]} we prove \eqref{Pclaim}. 

As a consequence we observe that there is a constant $C(s)>0$ such that for $1 \leq \ind \leq s-1$, 
\begin{equation}\label{Pclaim(2)}
 P_\ind([\bj]_{\ind-1},T) \leq C(s) T^{(s-1)/s}.
\end{equation}
According to \eqref{alpha-beta-def}, this leads to the estimate
\begin{equation}
 \beta_\ind(s,T) \leq C(s)\gamma_\ind(s) T^{(s-1)/s}, \qquad
\gamma_\ind(s) = \sum_{j_1=\ind}^{j_0-1} \dots \sum_{j_{\ind-1}=2}^{j_{\ind-2}-1}
\left[ \binom{j_0}{j_1} \dots \binom{j_{\ind-2}}{j_{\ind-1}}\binom{j_{\ind-1}}{1}\right]
\end{equation}
which together with \eqref{sumEBlow} implies
\begin{equation}
E_\ind - B_\ind \geq T\alpha_\ind(s) - \frac{1}{2}C(s) \gamma_\ind(s) T^{(s-1)/s}.
\end{equation}
Obviously, the first term in the right-hand side of the estimate dominates. The coefficients $c,\alpha_\ind$ and $\gamma_\ind$ are independent of $T$. Therefore, for a fixed $s$ there exists $T_* = T_*(s,a)$ such that for all $T > T_*(s,a)$ and all $1 \leq \ind \leq s-1$ it holds
\begin{equation}
E_\ind - B_\ind >0.
\end{equation}
Summation over $\ind$ and \eqref{I-Gamma=sum} provide \eqref{Gamma-asy(1)}.
\endproof


\end{document}